\documentclass[11pt]{amsart}
\usepackage{amsmath,amssymb,xy,array}
\usepackage[T1]{fontenc}
\usepackage{amsfonts}
\usepackage{amsmath}
\usepackage{amssymb}
\usepackage{amsthm}
\usepackage[cp850]{inputenc}
\usepackage{hyperref}
\usepackage{graphicx}
\usepackage{fancyhdr}
\usepackage{tikz}
\usepackage{longtable}
\usepackage{geometry}
\usepackage{textcomp}
\usetikzlibrary{trees}

\newcommand{\arXiv}[1]{\href{http://arxiv.org/abs/#1}{arXiv:#1}}
\def\bibaut#1{{\sc #1}}

\providecommand{\U}[1]{\protect\rule{.1in}{.1in}}
\providecommand{\U}[1]{\protect\rule{.1in}{.1in}}
\providecommand{\U}[1]{\protect\rule{.1in}{.1in}}
\providecommand{\U}[1]{\protect\rule{.1in}{.1in}}
\providecommand{\U}[1]{\protect\rule{.1in}{.1in}}
\input{xy}
\xyoption{all}
\usepackage{amsfonts}
\usepackage{amssymb,amsthm}
\usepackage{verbatim}
\usepackage{hyperref} 
\usepackage[T1]{fontenc}
\usepackage{amsmath}
\usepackage{enumerate}
\usepackage{tikz}
\usepackage{color}

\newcommand{\C}{\mathbb C}
\newcommand{\G}{\mathbb G}
\renewcommand{\P}{\mathbb P}

\DeclareMathOperator{\Sec}{Sec}
\DeclareMathOperator{\Stab}{Stab}

\newcommand{\QED}{\ifhmode\unskip\nobreak\fi\quad {\rm Q.E.D.}} % QED

\newcommand{\f}{\varphi}

\newcommand{\N}{\mathbb{N}}
 
\renewcommand{\P}{\mathbb{P}}

\renewcommand{\sec}{\mathbb{S}ec}
\DeclareMathOperator{\expdim}{expdim}

\DeclareMathOperator{\sign}{sign}

\newtheorem{Theorem}{Theorem}[section]
\newtheorem*{theoremn}{Theorem}

\newtheorem*{Conjecture}{Conjecture}
\newtheorem{Lemma}[Theorem]{Lemma}
\newtheorem{Proposition}[Theorem]{Proposition}

\newtheorem{Corollary}[Theorem]{Corollary}

\newtheorem{Assumption}[Theorem]{Assumption}

\theoremstyle{definition}

\newtheorem{Definition}[Theorem]{Definition}

\newtheorem{Example}[Theorem]{Example}

\hypersetup{pdfpagemode=UseNone}
\hypersetup{pdfstartview=FitH}
				
\begin{document}

\title{Non-Secant Defectivity via Osculating Projections}

\author[Alex Massarenti]{Alex Massarenti}
\address{\sc Alex Massarenti\\
Universidade Federal Fluminense (UFF)\\
Rua M\'ario Santos Braga\\
24020-140, Niter\'oi, Rio de Janeiro\\ 
Brazil}
\email{alexmassarenti@id.uff.br}

\author[Rick Rischter]{Rick Rischter}
\address{\sc Rick Rischter\\
Universidade Federal de Itajub\'a (UNIFEI)\\ 
Av. BPS 1303, Bairro Pinheirinho\\ 
37500-903, Itajub\'a, Minas Gerais\\ 
Brazil}
\email{rischter@unifei.edu.br}

\date{\today}

\maketitle

\begin{abstract}
We introduce a method to produce bounds for the non secant defectivity of an arbitrary irreducible projective variety, once we know how its osculating spaces behave in families and when the linear projections from them are generically finite.\\
Then we analyze the relative dimension of osculating projections of Grassmannians, and as an application of our techniques we prove that asymptotically the Grassmannian $\G(r,n)$, parametrizing $r$-planes in $\mathbb{P}^n$, is not $h$-defective for $h\leq (\frac{n+1}{r+1})^{\lfloor\log_2(r)\rfloor}$. This bound improves the previous one $h\leq \frac{n-r}{3}+1$, due to H. Abo, G. Ottaviani and C. Peterson, for any $r\geq 4$. 
\end{abstract}
		
\setcounter{tocdepth}{1}
\tableofcontents

\section*{Introduction}
The \textit{$h$-secant variety} $\sec_{h}(X)$, of an irreducible, non-degenerate $n$-dimensional variety $X\subset\mathbb{P}^N$, is the Zariski closure of the union of the linear spaces spanned by collections of $h$ points on $X$. Secant varieties are central objects in both classical algebraic geometry \cite{CC01}, \cite{Za}, and applied mathematics \cite{La12}, \cite{LM}, \cite{LO}, \cite{MR}.

The \textit{expected dimension} of $\sec_{h}(X)$ is
$$\expdim(\sec_{h}(X)):= \min\{nh+h-1,N\}.$$
However, the actual dimension of $\sec_{h}(X)$ might be smaller than the expected one. Indeed, this happens when trough a general point of $\mathbb{P}^N$ there are infinitely many $(h-1)$-planes $h$-secant to $X$. According to a definition by F. Zak \cite{Za} we will say that $X$ is \textit{$h$-defective} if 
$$\dim(\sec_{h}(X)) < \expdim(\sec_{h}(X)).$$

In this paper we introduce a method to produce bounds for the non secant defectivity of an arbitrary irreducible projective variety, based on the behavior of its osculating spaces and of the corresponding osculating projections. Then, with these techniques, we study the dimension of secant varieties of the Grassmannian $\G(r,n)$ parametrizing $r$-planes in $\mathbb{P}^n$. Grassmannians together with Veronese and Segre varieties form the triad of varieties parametrizing rank one tensors. Hence, a general point of their $h$-secant variety corresponds to a tensor of a given rank depending on $h$. For this reason, secant varieties of Grassmannians, Veroneses and Segres are particularly interesting in problems of tensor decomposition \cite{CM}, \cite{CGLM}, \cite{La12}, \cite{Me06}, \cite{Me09}, \cite{GM16}.

Furthermore, secant varieties have been widely used to construct and study moduli spaces for all  possible additive decompositions of a general tensor into a given number of rank one tensors \cite{Do04}, \cite{DK93}, \cite{Ma16}, \cite{MM13}, \cite{RS00}, \cite{TZ11}, \cite{BGI11}.

The problem of determining the actual dimension of secant varieties, and its relation with the dimension of certain linear systems of hypersurfaces with double points, have a very long history in algebraic geometry, and can be traced back to the Italian school \cite{Ca37}, \cite{Sc08}, \cite{Se01}, \cite{Te11}.

Since then the geometry of secant varieties has been studied and used by many authors in various contexts \cite{CC01}, \cite{CR06}, \cite{IR08}, \cite{Ru08}, and the problem of secant defectivity has been widely studied for Veroneses, Segres and Grassmannians \cite{AH95}, \cite{AB13}, \cite{AOP09a}, \cite{AOP09b}, \cite{Bo13}, \cite{CGG03},\cite{CGG05}, \cite{CGG11}, \cite{LP13}, \cite{BBC12}, \cite{BCC11}.

Despite the long history of this subject, only in $1995$ J. Alexander and A. Hirshowitz \cite{AH95} classified secant defective Veronese varieties. Indeed, they proved that, except for the double Veronese embedding which is almost always defective, the degree $d$ Veronese embedding of $\mathbb{P}^n$ is not $h$-defective, with the following exceptions: 
$$(d,n,h)\in\{(4,2,5),(4,3,9),(3,4,7),(4,4,14)\}.$$
Later on, K. Baur, J. Draisma, W. A. de Graaf proposed a conjecture on secant defectivity of Grassmannians in the spirit of Alexander-Hirshowitz result \cite{BDdG07}. 

It is well-known that the secant variety $\sec_h(\G(1,n))$, that is the locus of skew-symmetric matrices of rank at most $2h$, is almost always defective. Therefore, throughout the paper we assume $r\geq 2$. Only four defective cases are known
then, and we have the following conjecture.
\begin{Conjecture}\cite[Conjecture 4.1]{BDdG07}
If $r\geq 2$ then $\G(r,n)$ is not $h$-defective with the following exceptions:
$$(r,n,h)\in\{(2,7,3),(3,8,3),(3,8,4),(2,9,4)\}.$$
\end{Conjecture}
In \cite{CGG05} M. V. Catalisano, A. V. Geramita, and A. Gimigliano gave explicit bounds on $(r,n,h)$ for $\G(r,n)$ not to be $h$-defective. Later, in \cite{AOP09b} H. Abo, G. Ottaviani, and C. Peterson, improved these bounds, and showed that the conjecture is true for $h\leq 6$. Finally, in \cite{Bo13} A. Boralevi further improved this result by proving the conjecture for $h\leq 12$.

To the best of our knowledge, the best asymptotic bound for $\sec_h(\G(r,n))$ to have expected dimension was obtained by H. Abo, G. Ottaviani, and C. Peterson using monomial techniques.
\begin{theoremn}\cite[Theorem 3.3]{AOP09b}
If $r\geq 2$ and 
$$h\leq \frac{n-r}{3}+1$$
then $\sec_h(\G(r,n))$ has the expected dimension.
\end{theoremn}

Our starting point, in order to study the dimension of $\sec_h(\G(r,n))$, is a result due to L. Chiantini and C. Ciliberto \cite[Proposition 3.5]{CC01} relating secant defectivity with the dimension of the general fiber of a general tangential projection. Given $x_1,\dots,x_h\in X\subset\mathbb{P}^N$ general points, we may consider a \textit{general $h$-tangential projection} of $X$
$$\tau_{X,h}:X\subset\mathbb{P}^N\dasharrow\mathbb{P}^{N_h}$$
that is the linear projection with center $\left\langle T_{x_1}X,\dots,T_{x_h}X\right\rangle$. Then, by \cite[Proposition 3.5]{CC01} if $\tau_{X,h}$ is generically finite then $X$ is not $(h+1)$-defective.

Our approach consists in considering linear projections from higher order osculating spaces. If $p\in X\subset\mathbb{P}^N$ is a smooth point, the \textit{$m$-osculating space} $T_p^{m}X$ of $X$ at $p$ is essentially the linear subspace of $\mathbb{P}^N$ generated by the partial derivatives of order less or equal than $m$ of a local parametrization of $X$ at $p$, see Definition \ref{oscdef}.

Given $p_1,\dots, p_l\in X$ general points, we denote by 
$$\Pi_{T^{k_1,\dots,k_l}_{p_1,\dots,p_l}}:X\subset\mathbb{P}^N\dasharrow\mathbb{P}^{N_{k_1,\dots,k_l}}$$
the corresponding \textit{$(k_1+\dots +k_l)$-osculating projection}, that is the linear projection with center $\left\langle T_{p_1}^{k_1}X,\dots, T_{p_l}^{k_l}X\right\rangle$.

When $X = \G(r,n)$ we manage to control the dimension of the general fiber of osculating projections from the span of a certain number of general osculating spaces. Indeed, in Corollary \ref{oscprojbirationalII} we prove that, under suitable numerical hypothesis, such an osculating projection is birational.

Then, in Section \ref{degtanosc} we construct flat degenerations of general tangential projections to linear projections which factor through suitable osculating projections. Since, by Proposition \ref{p2}, the dimension of the general fiber can only increase under specialization, the birationality of a certain osculating projection yields that the general tangential projection degenerating to it is generically finite. 

However, we are not able to check if the degree of the map is preserved, that is if the general tangential projection is birational as well. In order to do this, one needs to achieve a good control on the indeterminacy locus of the relevant tangential projection.    

On the other hand, by \cite[Proposition 3.5]{CC01} knowing that a general tangential projection is generically finite is enough to conclude that, under the numerical hypothesis ensuring the birationality of the corresponding general osculating projection, $\G(r,n)$ is not defective. As a direct consequence of our main results in Theorem \ref{maingrass} we get the following.

\begin{theoremn}
Assume that $r\geq 2$, set 
$$\alpha:=\left\lfloor \dfrac{n+1}{r+1} \right\rfloor$$
and write $r = 2^{\lambda_1}+\dots+2^{\lambda_s}+\varepsilon$, with $\lambda_1 > \lambda_2 >\dots >\lambda_s\geq 1$, $\varepsilon\in\{0,1\}$. If either
\begin{itemize}
\item[-] $h\leq (\alpha-1)(\alpha^{\lambda_1-1}+\dots+\alpha^{\lambda_s-1})+1$ or
\item[-] $n\geq r^2+3r+1$ and $h\leq \alpha^{\lambda_1}+\dots+\alpha^{\lambda_s}+1$
\end{itemize}
then $\G(r,n)$ is not $h$-defective.
\end{theoremn}
Note that the bounds in our main result gives that asymptotically the Grassmannian $\G(r,n)$ is not $(\frac{n+1}{r+1})^{\lfloor\log_2(r)\rfloor}$-defective, while \cite[Theorem 3.3]{AOP09b} yields that $\G(r,n)$ is not $\frac{n}{3}$-defective. In Section \ref{uglymath} we show that Theorem \ref{maingrass} improves \cite[Theorem 3.3]{AOP09b} for any $r\geq 4$. However, H. Abo, G. Ottaviani, and C. Peterson in \cite{AOP09b} gave a much better bound, going with $n^2$, in the case $r=2$. 

We would like to mention that, as remarked by C. Ciliberto and F. Russo in \cite{CR06}, the idea that the behavior of osculating projections reflects the geometry of the variety itself was already present in the work of G. Castelnuovo \cite[Pages 186-188]{Ca37}.

Finally, we would like to stress that the machinery introduced in this paper could be used to produce bounds, for the non secant defectivity of an arbitrary irreducible projective variety, once we know how its osculating spaces behave in families and when the projections from them are generically finite. Indeed, in a forthcoming paper we will apply these techniques to Segre-Veronese varieties.

The paper is organized as follows. In Section \ref{secant} we recall some notions on secant varieties and tangential projections. In Section \ref{osculating} we compute explicitly the osculating spaces of Grassmannians, and in Section \ref{projosc} we study the relative dimension of general osculating projections. In Section \ref{degtanosc}, in order to extend our results on osculating projections to tangential projections, we investigate how rational maps degenerate in a $1$-dimensional family. Finally, in Section \ref{grassnodef} we take advantage of these techniques to prove our main result on the dimension of secant varieties of Grassmannians.

\subsection*{Acknowledgments}
We would like to thank Ciro Ciliberto for his useful comments, particularly about Section \ref{lls}. We would also like to thank Ada Boralevi, Luca Chiantini and Giorgio Ottaviani for helpful discussions. Finally, we thank Carolina Araujo for carefully reading and helping us improving a preliminary version of the paper.\\
The first named author is a member of the Gruppo Nazionale per le Strutture Algebriche, Geometriche e le loro Applicazioni of the Istituto Nazionale di Alta Matematica "F. Severi" (GNSAGA-INDAM).
The second named author would like to thanks CNPq for the financial support.

\section{Secant Varieties}\label{secant}
Throughout the paper we work over the field of complex numbers. In this section we recall the notions of secant varieties, secant defectivity and secant defect. We refer to \cite{Ru03} for a nice and comprehensive survey on the subject.

Let $X\subset\P^N$ be an irreducible non-degenerate variety of dimension $n$ and let
$$\Gamma_h(X)\subset X\times \dots \times X\times\G(h-1,N)$$
be the closure of the graph of the rational map
$$\alpha: X\times  \dots \times X \dasharrow \G(h-1,N),$$
taking $h$ general points to their linear span $\langle x_1, \dots , x_{h}\rangle$. Observe that $\Gamma_h(X)$ is irreducible and reduced of dimension $hn$. Let $\pi_2:\Gamma_h(X)\to\G(h-1,N)$ be the natural projection. We denote
$$\mathcal{S}_h(X):=\pi_2(\Gamma_h(X))\subset\G(h-1,N).$$
Again $\mathcal{S}_h(X)$ is irreducible and reduced of dimension $hn$. Finally, let
$$\mathcal{I}_h=\{(x,\Lambda) \: | \: x\in \Lambda\} \subset\P^N\times\G(h-1,N)$$
with natural projections $\pi_h$ and $\psi_h$ onto the factors. Furthermore, observe that $\psi_h:\mathcal{I}_h\to\G(h-1,N)$ is a $\P^{h-1}$-bundle on $\G(h-1,N)$.

\begin{Definition} Let $X\subset\P^N$ be an irreducible non-degenerate variety. The {\it abstract $h$-secant variety} is the irreducible variety
$$\Sec_{h}(X):=(\psi_h)^{-1}(\mathcal{S}_h(X))\subset \mathcal{I}_h.$$
The {\it $h$-secant variety} is
$$\sec_{h}(X):=\pi_h(Sec_{h}(X))\subset\P^N.$$
It immediately follows that $\Sec_{h}(X)$ is a $(hn+h-1)$-dimensional variety with a $\P^{h-1}$-bundle structure over $\mathcal{S}_h(X)$. We say that $X$ is \textit{$h$-defective} if $$\dim\sec_{h}(X)<\min\{\dim\Sec_{h}(X),N\}.$$
The number 
$$\delta_h(X) = \min\{\dim\Sec_{h}(X),N\}-\dim\sec_{h}(X)$$ 
is called the \textit{$h$-defect} of $X$. We say that $X$ is \textit{$h$-defective} if $\delta_{h}(X) > 0$.
\end{Definition}

Now, let $x_1,\dots,x_h\in X\subset\mathbb{P}^N$ be general points, and let $T_{x_i}X$ be the tangent space of $X$ at $x_i$. We will call the linear projection
$$\tau_{X,h}:X\subseteq\mathbb{P}^N\dasharrow\mathbb{P}^{N_h}$$
with center $\left\langle T_{x_1}X,\dots,T_{x_h}X\right\rangle$ a \textit{general $h$-tangential projection} of $X$. Finally, let $X_h = \tau_{X,h}(X)$. We will need the first part of the following result due to L. Chiantini and C. Ciliberto.

\begin{Proposition}\cite[Proposition 3.5]{CC01}\label{cc}
Let $X\subset\mathbb{P}^N$ be a irreducible, non-degenerate, projective variety of dimension $n$.
\begin{itemize}
\item[-] If $\dim(X_h) = \dim(X)$, that is $\tau_{X,h}:X\dasharrow X_h$ is generically finite, then $X$ is not $(h+1)$-defective.
\item[-] If $N-\dim(\left\langle T_{x_1}X,\dots,T_{x_h}X\right\rangle)-1\geq n$ and $\dim(X_h)<\dim(X)$, that is $\tau_{X,h}:X\dasharrow X_h$ has positive dimensional general fibers, then $X$ is $(h+1)$-defective.
\end{itemize}
\end{Proposition}

For instance, let $\nu_2^n:\mathbb{P}^n\rightarrow\mathbb{P}^{N_n}$ be the $2$-Veronese embedding of $\mathbb{P}^n$, with $N_n = \frac{1}{2}(n+2)(n+1)-1$, $X = V^n_2\subset\mathbb{P}^{N_n}$ the corresponding Veronese variety, and $x_1,\dots, x_{h}\in V^n_2$ general points, with $h\leq n-1$. The linear system of hyperplanes in $\mathbb{P}^{N_n}$ containing $\left\langle T_{x_1}V^n_2,\dots, T_{x_{h}}V^n_2\right\rangle$ corresponds to the linear system of quadrics in $\mathbb{P}^n$ whose vertex contains $\Lambda=\left\langle\nu_2^{-1}(x_1),\dots, \nu_2^{-1}(x_{h})\right\rangle$. Therefore, we have the following commutative diagram
\[
  \begin{tikzpicture}[xscale=2.7,yscale=-1.2]
    \node (A0_0) at (0, 0) {$\mathbb{P}^n$};
    \node (A0_1) at (1, 0) {$V_2^n\subset\mathbb{P}^{N_n}$};
    \node (A1_0) at (0, 1) {$\mathbb{P}^{n-h}$};
    \node (A1_1) at (1, 1) {$V_2^{n-h}\subset\mathbb{P}^{N_{n-h}}$};
    \path (A0_0) edge [->]node [auto] {$\scriptstyle{\nu_2^n}$} (A0_1);
    \path (A0_0) edge [->,swap,dashed]node [auto] {$\scriptstyle{\pi_\Lambda}$} (A1_0);
    \path (A0_1) edge [->,dashed]node [auto] {$\scriptstyle{\tau_{X,h}}$} (A1_1);
    \path (A1_0) edge [->]node [auto] {$\scriptstyle{\nu_2^{n-h}}$} (A1_1);
  \end{tikzpicture}
  \]
where $\pi_{\Lambda}:\mathbb{P}^n\dasharrow\mathbb{P}^{n-h}$ is the projection form $\Lambda$. Hence $\tau_{X,h}$ has positive relative dimension, and Proposition \ref{cc} yields, as it is well-known, that $V_2^n$ is $h$-defective for any $h\leq n$.
 
\section{Osculating Spaces of Grassmannians}\label{osculating}
Let $X\subset \P^N$ be an integral projective variety of dimension $n$, $p\in X$ a smooth point, and 
$$
\begin{array}{cccc}
\phi: &\mathcal{U}\subseteq\mathbb{C}^n& \longrightarrow & \mathbb{C}^{N}\\
      & (t_1,\dots,t_n) & \longmapsto & \phi(t_1,\dots,t_n)
\end{array}
$$
with $\phi(0)=p$, be a local parametrization of $X$ in a neighborhood of $p\in X$. 

For any $m\geq 0$ let $O^m_pX$ be the affine subspace of $\mathbb{C}^{N}$ passing through $p\in X$, and whose direction is given by the subspace generated by the vectors $\phi_I(0)$, where $I = (i_1,\dots,i_n)$ is a multi-index such that $|I|\leq m$ and 
\begin{equation}\label{osceq}
\phi_I = \frac{\partial^{|I|}\phi}{\partial t_1^{i_1}\dots\partial t_n^{i_n}}.
\end{equation}

\begin{Definition}\label{oscdef}
The $m$-\textit{osculating space} $T_p^m X$ of $X$ at $p$ is the projective closure in $\mathbb{P}^N$ of the affine subspace $O^m_pX\subseteq \mathbb{C}^{N}$.
\end{Definition}

For instance, $T_p^0 X=\{p\}$, and $T_p^1 X$ is the usual tangent space of $X$ at $p$. When no confusion arises we will write $T_p^m$ instead of $T_p^mX$.

Osculating spaces can be defined intrinsically. Let $\mathcal{L}$ be an invertible sheaf on $X$, $V = \mathbb{P}(H^0(X,\mathcal{L}))$, and $\Delta\subset X\times X$ the diagonal. The rank $\binom{n+m}{m}$ locally free sheaf
$$J_m(\mathcal{L}) = \pi_{1*}(\pi_{2}^{*}(\mathcal{L})\otimes \mathcal{O}_{X\times X}/\mathcal{I}_{\Delta}^{m+1})$$
is called the \textit{$m$-jet bundle} of $\mathcal{L}$. Note that the fiber of $J_m(\mathcal{L})$ at $p\in X$ is 
$$J_m(\mathcal{L})_p\cong H^0(X,\mathcal{L}\otimes\mathcal{O}_X/\mathfrak{m}_p^{m+1})$$
and the quotient map 
$$j_{m,p}:V\rightarrow H^0(X,\mathcal{L}\otimes\mathcal{O}_X/\mathfrak{m}_p^{m+1})$$
is nothing but the evaluation of the global sections and their derivatives of order at most $m$ at the point $p\in X$. Let 
$$j_m:V\otimes\mathcal{O}_X\rightarrow J_m(\mathcal{L})$$
be the corresponding vector bundle map. Then, there exists an open subset $U_m\subseteq X$ where $j_m$ is of maximal rank $r_m\leq \binom{n+m}{m}$.

The linear space $\mathbb{P}(j_{m,p}(V)) = T_p^m X\subseteq\mathbb{P}(V)$ is the $m$-\textit{osculating space} of $X$ at $p\in X$. The integer $r_m$ is called the \textit{general $m$-osculating dimension} of $\mathcal{L}$ on $X$.

Note that while the dimension of the tangent space at a smooth point is always equal to the dimension of the variety, higher order osculating spaces can be strictly smaller than expected even at a general point. In general, we have
\begin{equation}\label{dimosc}
\dim(T_p^m X) = \min\left\{\binom{n+m}{n}-1-\delta_{m,p},N\right\}
\end{equation}
where $\delta_{m,p}$ is the number of independent differential equations of order less or equal than $m$ satisfied by $X$ at $p$.

Projective varieties having general $m$-osculating dimension smaller than expected were introduced and studied in \cite{Seg07}, \cite{Te12}, \cite{Bom19}, \cite{To29}, \cite{To46}, and more recently in \cite{PT90}, \cite{BPT92}, \cite{BF04}, \cite{MMRO13}, \cite{DiRJL15}.

In particular, these works highlight how algebraic surfaces with defective higher order osculating spaces contain many lines, such as rational normal scrolls, and developable surfaces, that is cones or tangent developables of curves. As an example, which will be useful later on in the paper, we consider tangent developables of rational normal curves.

\begin{Proposition}\label{tdrnc}
Let $C_n\subseteq\mathbb{P}^n$ be a rational normal curve of degree $n$ in $\mathbb{P}^n$, and let $Y_n\subseteq\mathbb{P}^n$ be its tangent developable. Then
$$\dim(T^m_pY_n) = \min\{m+1,n\}$$
for $p\in Y_n$ general, and $m\geq 1$.
\end{Proposition}
\begin{proof}
We may work on an affine chart. Then $Y_n$ is the surface parametrized by
$$
\begin{array}{cccc}
\phi: & \mathbb{A}^2 & \longrightarrow & \mathbb{A}^n\\ 
 & (t,u) & \mapsto & (t+u,t^2+2tu,\dots ,t^n+nt^{n-1}u)
\end{array} 
$$
Note that 
$$\frac{\partial^m \phi}{\partial t^{m-k} \partial u^k}=0$$
for any $k\geq 2$. Furthermore, we have
$$\frac{\partial^m\phi}{\partial t^m}-\frac{\partial^m\phi}{\partial t^{m-1}\partial u} = u\frac{\partial^{m+1}\phi}{\partial t^m\partial u}$$
for any $m\geq 1$.

Therefore, for any $m\geq 1$ we get just two non-zero partial derivatives of order $m$, and one partial derivative is given in terms of smaller order partial derivatives. Furthermore, in the notation of (\ref{dimosc}) we have $\delta_{m,p} = \frac{m(m+1)}{2}-1$ for any $1\leq m\leq n-1$, where $p\in Y_n$ is a general point.
\end{proof}

From now on we will assume that $n \geq 2r+1.$ We will denote by $e_0,\dots,e_n\in \C^{n+1}$ both the vectors of the canonical basis of $\mathbb{C}^{n+1}$ and the corresponding points in $\P^n=\P(\C^{n+1})$.

Throughout the paper we will always view $\G(r,n)$ as a projective variety in its Pl\"ucker embedding, that is the morphism induced by the determinant of the universal quotient bundle $\mathcal{Q}_{\G(r,n)}$ on $\G(r,n)$:
$$
\begin{array}{cccc}
\f_{r,n}: &\G(r,n)& \longrightarrow & \P^N:=\P(\bigwedge^{r+1}\C^{n+1})\\
      & \left\langle v_0,\dots,v_r\right\rangle & \longmapsto & [v_0\wedge \dots\wedge v_r]
\end{array}
$$
where $N = \binom{n+1}{r+1}-1$. Now, let 
$$\Lambda:=\left\{ I\subset \{0,\dots,n\}, |I|=r+1 \right\}.$$
For each $I=\{i_0,\dots,i_r\}\in \Lambda$ let $e_I\in\G(r,n)$ be the point corresponding to $e_{i_0}\wedge\dots\wedge e_{i_r}\in\bigwedge^{r+1} \C^{n+1}$. 

Furthermore, we define a distance on $\Lambda$ as 
$$d(I,J)=|I|-|I\cap J|=|J|-|I\cap J|$$ 
for each $I,J\in\Lambda$. Note that, with respect to this distance, the diameter of $\Lambda$ is $r+1$.

In the following we give an explicit description of osculating spaces of Grassmannians at fundamental points.

\begin{Proposition}\label{oscgrass}
For any $s\geq 0$ we have 
$$T^s_{e_I}(\G(r,n))= \left\langle e_J \: | \: d(I,J)\leq s\right\rangle = \{p_J=0 \: | \: d(I,J)>s\}\subseteq\P^N.$$
In particular, $T^s_{e_I}(\G(r,n))=\P^N$ for any $s\geq r+1$.
\end{Proposition}
\begin{proof}
We may assume that $I\!=\!\{0,\dots,r\}$ and consider the usual parametrization of $\G(r,n):$ 
$$\phi:\C^{(r+1)(n-r)}\rightarrow \G(r,n)$$
given by 
$$
A=(a_{ij})=\begin{pmatrix}
1& \dots & 0 &a_{0,r+1}& \dots & a_{0n}\\
\vdots & \ddots & \vdots & \vdots &   \ddots &\vdots\\
0 & \dots & 1 & a_{r,r+1} & \dots & a_{rn}\\
\end{pmatrix}
\mapsto (\det (A_J))_{J\in \Lambda}
$$
where $A_J$ is the $(r+1)\times(r+1)$ matrix obtained from $A$ considering just the columns indexed by $J$.

Note that each variable appears in degree at most one in the coordinates of $\phi$. Therefore, deriving two times with respect to the same variable always gives zero.

Thus, in order to describe the osculating spaces we may take into account just partial derivatives
with respect to different variables. Moreover, since the degree of $\det (A_J)$ with respect to $a_{i,j}$ is at most $r+1$ all partial derivatives of order greater or equal than $r+2$ are zero. Hence, it is enough to prove the proposition for $s\leq r+1$.

Given $J=\{j_0,\dots,j_r\}\subset \{0,\dots,n\}$, $k\in \{0,\dots,r\}$, and $k'\in \{r+1,\dots,n\}$ we have
$$ \dfrac{\partial \det (A_J)}{\partial a_{k,k'}}=\begin{cases}
0 &\mbox{ if } k'\notin J\\
(-1)^{l+1+k'} \det (A_{J,k,k'}) &\mbox{ if } k'=j_l
\end{cases}$$
where $A_{J,k,k'}$ denotes the submatrix of $A_{J}$ obtained deleting the line indexed by $k$ and the column indexed by $k'$. More generally, for any $m\geq 1$ and for any
\begin{align*}
J&=\{j_0,\dots,j_r\}\subset \{0,\dots,n\},\\
K'&=\{k_1',\dots,k_m'\}\subset \{r+1,\dots,n\}, \\
K&=\{k_1,\dots,k_m\}\subset \{0,\dots,r\}
\end{align*}
we have
$$ \dfrac{\partial^m \det (A_J)}{\partial a_{k_1,k_1'}\dots\partial a_{k_m,k_m'}}=
\begin{cases}
(\pm 1) \det (A_{J,(k_1,k_1'),\dots,(k_m,k_m')} )
&\mbox{ if }K'\!\subset\! J \mbox{ and } |K|\!=\!|K'|\!=\!m\leq d\\
0 &\mbox{ otherwise }
\end{cases}$$
where $d=d(J,\{0,\dots,r\})=\deg(\det(A_J))$. Therefore
$$ \dfrac{\partial^m \det (A_J)}{\partial a_{k_1,k_1'}\dots\partial a_{k_m,k_m'}}(0)=\begin{cases}
\pm 1 &\mbox{ if } J=K'\bigcup \left(\{ 0,\dots, r\}\backslash K\right)\\
0 &\mbox{ otherwise }
\end{cases}$$
and
$$\dfrac{\partial^m \phi}{\partial a_{k_1,k_1'}\dots\partial a_{k_m,k_m'}}(0)=
\pm  e_{K'\cup \left(\{ 0,\dots, r\}\setminus K\right)}.$$
Note that $d\left(K'\cup \left(\{ 0,\dots, r\}\backslash K\right),\{ 0,\dots, r\}\right)=m$, and that any $J$ with $d(J,\{0,\dots,r\})=m$ may be written in the form $K'\cup \left(\{ 0,\dots, r\}\backslash K\right)$.

Finally, we get that
$$\left\langle \dfrac{\partial^{|I|}\phi}{\partial^I a_{i,j}}(0)\: \big| \: |I|= m\right\rangle
=\left\langle e_J\: | \: d(J,\{0,\dots r\})=m\right\rangle$$
which proves the statement.
\end{proof}

Now, it is easy to compute the dimension of the osculating spaces of $\G(r,n)$.

\begin{Corollary}
For any point $p\in \G(r,n)$ we have
$$\dim T^s_p \G(r,n)=\sum_{l=1}^s \binom{r+1}{l}\binom{n-r}{l}$$
for any $0\leq s\leq r$, while $T^s_p \G(r,n)=\P^N$ for any $s\geq r+1$.
\end{Corollary}
\begin{proof}
Since $\G(r,n)\subset\P^{N}$ is homogeneous under the action the algebraic subgroup 
$$\Stab(\G(r,n))\subset PGL(N+1)$$
stabilizing it, there exists an automorphism $\alpha\in PGL(N+1)$ inducing an automorphism of $\G(r,n)$ such that $\alpha(p) = e_I$. Moreover, since $\alpha\in PGL(N+1)$ we have that it induces an isomorphism between $T^s_p \G(r,n)$ and $T^s_{e_I} \G(r,n)$. Now, the computation of $\dim \G(r,n)$ follows, by standard combinatorial computations, from Proposition \ref{oscgrass}.
\end{proof}

\section{Osculating projections}\label{projosc}
In this section we study linear projections of Grassmannians from their osculating spaces. In order to help the reader get acquainted with the ideas of the proofs, we start by studying in detail projections from a single osculating space. 

Let us denote by $(p_I)_{I\in \Lambda}$ the Pl\"ucker coordinates on $\mathbb{G}(r,n)$, let
$0\leq s\leq r$ be an integer, and $I\in \Lambda$. By Proposition \ref{oscgrass} the projection of $\G(r,n)$ from $T_{e_I}^s$ is given by
\begin{align*}
\Pi_{T_{e_I}^s}:\G(r,n)&\dasharrow \P^{N_{s}}\\
(p_I)_{I\in \Lambda}
&\mapsto (p_J)_{J\in \Lambda \: |\: d(I,J)>s}
\end{align*}
Moreover, given $I'\!=\!\{i_0',\dots,i_s'\}\subset I$ with $|I'|\!=\! s+1$
we can consider the linear projection
\begin{align*}
\pi_{I'}:\P^n&\dasharrow \P^{n-s-1}\\
(x_i)
&\mapsto (x_i)_{i\in \{0,\dots,n\}\setminus I'}
\end{align*}
which in turn induces the linear projection
\begin{align*}
\Pi_{I'}:\G(r,n)&\dasharrow \G(r,n-s-1)\\
[V]&\mapsto [\pi_{I'}(V)]\\
(p_I)_{I\in \Lambda}
&\mapsto(p_J)_{J\in \Lambda \: | \: J\cap I'=\emptyset}
\end{align*}

Note that the fibers of $\Pi_{I'}$ are isomorphic to $\G(r,r+s+1)$. More precisely, let 
$y\in \G(r,n-s-1)$ be a point, and consider a general point $x\in \overline{\Pi_{I'}^{-1}(y)}\subset \G(r,n)$ 
corresponding to an $r$-plane $V_x\subset \P^n$. Then we have  
$$\overline{\Pi_{I'}^{-1}(y)}=\G\left(r,\left\langle  V_x, e_{i_0'},\dots,e_{i_s'}\right\rangle\right)\subset \G(r,n).$$

On the other hand, a priori it is not at all clear what are the fibers of $\Pi_{T_{e_I}^s}$. In general the image of $\Pi_{T_{e_I}^s}$ is very singular, and its fibers may not be connected. In what follows we study the general fiber of the map $\Pi_{T_{e_I}^s}$ by factoring it through several projections of type $\Pi_{I'}$. 

\begin{Lemma}\label{oscfactors}
If $s=0,\dots, r$ and $I'\subset I$ with $|I'|=s+1$, then the rational map $\Pi_{I'}$ factors through $\Pi_{T_{e_I}^s}$. Moreover, $\Pi_{T_{e_I}^r}=\Pi_I$.
\end{Lemma}
\begin{proof}
Since $J\cap I'=\emptyset\Rightarrow d(I,J)>s$ the center of $\Pi_{T_{e_I}^s}$ in contained in the center of $\Pi_{I'}$. Furthermore, if $s=r$ then $J\cap I=\emptyset\Leftrightarrow d(I,J)>r.$
\end{proof}

Now, we are ready to describe the fibers of $\Pi_{T_{e_I}^s}$ for $0\leq s\leq r$.

\begin{Proposition}\label{oscprojbirational}
The rational map $\Pi_{T_{e_I}^s}$ is birational for every $0\leq s\leq r-1$, and 
$$\Pi_{T_{e_I}^r}:\G(r,n)\dasharrow \G(r,n-r-1)$$ 
is a fibration with fibers isomorphic to $\G(r,2r+1)$.
\end{Proposition}
\begin{proof}
For the second part of the statement it is enough to observe that $\Pi_{T_{e_I}^r}=\Pi_I$. Now, let us consider the first claim. Since $\Pi_{T_{e_I}^s}$ factors through $\Pi_{T_{e_I}^{s-1}}$ it is enough to prove that $\Pi_{T_{e_I}^{r-1}}$ is birational. By Lemma \ref{oscfactors} for any $I_j = I\setminus\{i_j\}$, there exists a rational map $\tau_j$ such that the following diagram is commutative
\[
  \begin{tikzpicture}[xscale=3.5,yscale=-1.5]
    \node (A0_0) at (0, 0) {$\G(r,n)$};
    \node (A1_0) at (1, 1) {$\G(r,n-r)$};
    \node (A1_1) at (1, 0) {$W\subseteq\mathbb{P}^{N_s}$};
    \path (A0_0) edge [->,swap, dashed] node [auto] {$\scriptstyle{\Pi_{I_j}}$} (A1_0);
    \path (A1_1) edge [->, dashed] node [auto] {$\scriptstyle{\tau_j}$} (A1_0);
    \path (A0_0) edge [->, dashed] node [auto] {$\scriptstyle{\Pi_{T_{e_I}^{r-1}}}$} (A1_1);
  \end{tikzpicture}
\]
where $W=\overline{\Pi_{T_{e_I}^{r-1}}(\G(r,n))}$. Now, let $x\in W$ be a general point, and $F\subset\G(r,n)$ be the fiber of $\Pi_{T_{e_I}^{r-1}}$ over $x$. Set $x_j= \tau_j(x) \in \G(r,n-r)$, and denote by $F_j \subset \G(r,n)$ the fiber of $\Pi_{I_j}$ over $x_j$. Therefore 
\begin{equation}\label{intfib1}
F\subseteq \bigcap_{j=0}^{r} F_j.
\end{equation}
Now, note that if $y\in F$ is a general point corresponding to an $r$-plane $V_y\subset \P^n$ we have
$$F_j=\G(r,\left\langle V_y,e_{i_0},\dots,\widehat{e_{i_j}},\dots,e_{i_r}\right\rangle)$$ 
and hence
$$\bigcap_{j=0}^{r} F_j=\bigcap_{j=0}^{r}\G(r,\left\langle V_y,e_{i_0},\dots,\widehat{e_{i_j}},\dots,e_{i_r}\right\rangle)=\G(r,V_y)=\{y\}.$$
The last equality and (\ref{intfib1}) force $F=\{y\}$, and since we are working in characteristic zero $\Pi_{T_{e_I}^{r-1}}$ is birational.
\end{proof}

Our next aim is to study linear projections from the span of several osculating spaces. In particular, we want to understand when such a projection is birational as we did in Proposition \ref{oscprojbirational} for the projection from a single osculating space.

Clearly, there are some natural numerical constraints regarding how many coordinate points of $\G(r,n)$ we may take into account, and the order of the osculating spaces we want to project from.

First of all, by Proposition \ref{oscprojbirational} the order of the osculating spaces cannot exceed $r-1$. Furthermore, since in order to carry out the computations, we need to consider just coordinate points of $\G(r,n)$ corresponding to linearly independent linear subspaces of dimension $r+1$ in $\mathbb{C}^{n+1}$ we can use at most 
$$\alpha:=\left\lfloor \dfrac{n+1}{r+1} \right\rfloor$$ 
of them.

Now, let us consider the points $e_{I_1},\dots,e_{I_\alpha}\in\G(r,n)$ where 
\begin{equation}\label{I1Ialpha}
I_1=\{0,\dots,r\},\dots,I_\alpha=\{(r+1)(\alpha-1),\dots,(r+1)\alpha-1\}\in \Lambda.
\end{equation}

Again by Proposition \ref{oscgrass} the projection from the span of the osculating spaces of $\G(r,n)$ 
of orders $s_1,\dots,s_l$ at the points $e_{I_1},\dots,e_{I_l}$ is given by
\begin{align*}
\Pi_{T_{e_{I_1},\dots,e_{I_l}}^{s_1,\dots,s_l}}:\G(r,n)&\dasharrow \P^{N_{s_1,\dots, s_l}}\\
(p_I)_{I\in \Lambda}&\mapsto (p_J)_{J\in \Lambda\: | \: d(I_1,J)>s_1,\dots,d(I_l,J)>s_l}
\end{align*}
whenever $\{J\in \Lambda\: | \: d(I_1,J)\leq s_1 \mbox{ or } \dots \mbox{ or } d(I_l,J)\leq s_l\}\neq \Lambda$, and $l\leq \alpha$.

Furthermore, for any 
$I_1'=\{i_0^1,\dots,i_{s_1}^1\}\subset I_1,\dots,I_{l}'=\{i_0^l,\dots,i_{s_l}^l\}\subset I_l$ we consider the projection
\begin{align*}
\pi_{I_1',\dots,I_l'}:\P^n&\dasharrow \P^{n-l-\sum_1^l s_i}\\
(x_i)_{i=0,\dots,n}
&\mapsto (x_i)_{i\in\{0,\dots,n\}\setminus (I_1'\cup\dots\cup I_l')}
\end{align*}
where $l\leq \alpha$ and $n-l-\sum_1^l s_i\geq r+1$. The map $\pi_{I_1',\dots,I_l'}$ in turn induces the projection
\begin{align*}
\Pi_{I_1',\dots,I_l'}:\G(r,n)&\dasharrow \G\left(r,n-l-\textstyle\sum_1^l s_i\right)\\
[V]&\mapsto [\pi_{I_1',\dots,I_l'}(V)]\\
(p_I)_{I\in \Lambda}
&\mapsto(p_J)_{J\in \Lambda\: |\: J\cap \left(I_1'\cup\dots\cup I_l' \right)=\emptyset}
\end{align*}

\begin{Lemma}\label{oscfactorsII}
Let $I_1,\dots,I_\alpha$ be as in (\ref{I1Ialpha}), $l,s_1,\dots,s_l$ be integers such that $0\leq s_j\leq r-1$, and $0<l\leq \min\{\alpha,n-r-1-\sum_i s_i\}$. Then for any $I_1'\!=\!\{i_0^1,\dots,i_{s_1}^1\}\!\subset\! I_1,\dots, I_l'\!=\!\{i_0^l,\dots,i_{s_l}^l\}\!\subset\! I_l$ with $|I'_j|=s_j+1$ the rational maps $\Pi_{T_{e_{I_1},\dots,e_{I_l}}^{s_1,\dots,s_l}}$ and $\Pi_{I_1',\dots,I_l'}$ are well-defined and the latter factors through the former.
\end{Lemma}
\begin{proof}
Note that $J\cap \left(I_1'\cup\dots\cup I_l'\right)=\emptyset$ yields $d(I_1,J)>s_1,\dots,d(I_l,J)>s_l$.
Note also that the $I_j$'s are disjoint since the $I_j'$'s are.
Furthermore, since $\sum(s_i+1)=l+\sum s_i\leq n-r-1$ and $n\geq 2r+1,$ there are at least $r+2$ elements in $\{0,\dots,n\}\backslash \left(I_1'\cup\dots\cup I_l' \right)$. If $k_1,\dots,k_{r+2}$ are such elements, then 
$$K_j:=\{k_1,\dots,\widehat{k_j},\dots,k_{r+2}\}\in \{J\in \Lambda\: | \: d(I_j,J)> s_j,\: j=1,\dots,l\}$$
for any $j=1,\dots,r+2$ forces 
$\{J\in \Lambda \: | \: d(I_1,J)\leq s_1 \mbox{ or } \dots \mbox{ or } d(I_l,J)\leq s_l\}\neq \Lambda$.
\end{proof}

Now, we are ready to prove the main result of this section.

\begin{Proposition}\label{oscprojbir}
Let $I_1,\dots,I_\alpha$ be as in (\ref{I1Ialpha}), $l,s_1,\dots,s_l$ be integers such that $0\leq s_j\leq r-1$, and $0<l\leq \min\{\alpha,n-r-1-\sum_i s_i\}$. Then the projection $\Pi_{T_{e_{I_1},\dots,e_{I_l}}^{s_1,\dots,s_l}}$ is birational.
\end{Proposition}
\begin{proof}
For any collection of subsets $I_i'\subset I_i$ with $|I_i'|= s_i+1$ set $I'=\bigcup_i I_i'$. By Lemma \ref{oscfactorsII} there exists a rational map $\tau_{I_1',\dots,I_l'}$ fitting in the following commutative diagram 
\[
  \begin{tikzpicture}[xscale=3.5,yscale=-1.5]
    \node (A0_0) at (0, 0) {$\G(r,n)$};
    \node (A1_0) at (1, 1) {$\G(r,n-l-\sum_1^l s_i)$};
    \node (A1_1) at (1, 0) {$W\subseteq\mathbb{P}^{N_{s_1,\dots s_l}}$};
    \path (A0_0) edge [->,swap, dashed] node [auto] {$\scriptstyle{\Pi_{I_1',\dots,I_l'}}$} (A1_0);
    \path (A1_1) edge [->, dashed] node [auto] {$\scriptstyle{\tau_{I_1',\dots,I_l'}}$} (A1_0);
    \path (A0_0) edge [->, dashed] node [auto] {$\scriptstyle{\Pi_{T_{e_{I_1},\dots,e_{I_l}}^{s_1,\dots,s_l}}}$} (A1_1);
  \end{tikzpicture}
\]
where $W= \overline{\Pi_{T_{e_{I_1},\dots,e_{I_l}}^{s_1,\dots,s_l}}(\G(r,n))}$. Now, let $x\in W$ be a general point, and $F\subset \G(r,n)$ be the fiber of $\Pi_{T_{e_{I_1},\dots,e_{I_l}}^{s_1,\dots,s_l}}$ over $x$. Set $x'= \tau_{I_1',\dots,I_l'}(x) \in \G(r,n-l-\sum_1^l s_i)$, and denote by
$$F_{I_1',\dots,I_l'} \subset \G(r,n)$$
the fiber of $\Pi_{I_1',\dots,I_l'}$ over $x'$. Therefore
\begin{equation}\label{inc1}
F\subseteq \bigcap_{I_1',\dots,I_l'} F_{I_1',\dots,I_l'}
\end{equation}
where the intersection runs over all the collections of subsets $I_i'\subset I_i$ with $|I_i'|= s_i+1$. Now, if $y\in F$ is a general point corresponding to an $r$-plane $V_y\subset \P^n$ we have
$$F_{I_1',\dots,I_l'}=\G\left(r,\left\langle V_y,e_j \: | \: j\in I'\right\rangle\right)$$ 
and hence
\begin{equation}\label{inc2}
\bigcap_{I_1',\dots,I_l'} F_{I_1',\dots,I_l'}=
\bigcap_{I_1',\dots,I_l'}\G\left(r,\left\langle V_y,e_j \: | \: j\in I'\right\rangle\right)=\G(r,V_y)=\{y\}
\end{equation}
where again the first intersection is taken over all the subsets $I_i'\subset I_i$ with $|I_i'|= s_i+1$. 

Finally, to conclude it is enough to observe that (\ref{inc1}) and (\ref{inc2}) yield $F=\{y\}$, and since we are working in characteristic zero $\Pi_{T_{e_{I_1},\dots,e_{I_l}}^{s_1,\dots,s_l}}$ is birational.
\end{proof}

In what follows we just make Proposition \ref{oscprojbir} more explicit. 

\begin{Corollary}\label{oscprojbirationalII}
Set $\alpha:=\left\lfloor \dfrac{n+1}{r+1} \right\rfloor$
and let $I_1,\dots,I_\alpha$ be as in (\ref{I1Ialpha}).
Then $\Pi_{T_{e_{I_1},\dots,e_{I_{\alpha-1}}}^{r-1,\dots,r-1}}$ is birational. Furthermore, if $n\geq r^2+3r+1$ then $\Pi_{T_{e_{I_1},\dots,e_{I_\alpha}}^{r-1,\dots,r-1}}$ is birational.

Now, set $r':=n-2-\alpha r$ and $r'':=\min\{n-3-\alpha (r-1),r-2\}$.
If $2r+1<n<r^2+3r+1$ then 
\begin{itemize}
\item[-] $r-1\geq r'\geq 0$ and $\Pi_{T_{e_{I_1},\dots,e_{I_{\alpha-1}},e_{I_\alpha}}^{r-1,\dots,r-1,r'}}$ is birational;
\item[-] $r''\geq 0$ and $\Pi_{T_{e_{I_1},\dots,e_{I_{\alpha-1}},e_{I_\alpha}}^{r-2,\dots,r-2,r''}}$ is birational.
\end{itemize}
\end{Corollary}
\begin{proof}
First we apply Proposition \ref{oscprojbir} with $l=\alpha-1$ and $s_1=\dots=s_{\alpha-1}=r-1$. In this case the constraint is $\alpha-1\leq n-r-1-(\alpha-1)(r-1)$, that is $\alpha\leq \dfrac{n-r-1}{r}+1$. Note that this is always the case since
$$\alpha\leq\dfrac{n+1}{r+1}\leq \dfrac{n-1}{r}=\dfrac{n-r-1}{r}+1.$$

If $l=\alpha$ and $s_1=\dots=s_{\alpha}=r-1$
the constraint in Proposition \ref{oscprojbir} is $\alpha\leq n-r-1-\alpha(r-1)$, which is equivalent to $\alpha\leq \dfrac{n-r-1}{r}$. Now, it is enough to observe that 
$$\dfrac{n+1}{r+1}\leq\dfrac{n-r-1}{r} \Longleftrightarrow n\geq r^2+3r+1.$$

If $n\geq r^2+3r+1,$ then the claim follows from the inequalities $\alpha\leq \dfrac{n+1}{r+1}\leq \dfrac{n-r-1}{r}$.

Now assume that $n< r^2+3r+1$. First we check that $r'=n-2-\alpha r\leq r-1,$ that is 
$\alpha \geq \dfrac{n-1-r}{r}.$ That follows from
$$\alpha\geq \dfrac{n+1}{r+1}-1= \dfrac{n-r}{r+1}
\geq\dfrac{n-r-1}{r}$$
whenever $n\geq 2r+1.$
Next we check that $r',r''\geq 0.$
If $2r+1<n< 3r+2$ then $\alpha=2$, and $r'=n-2-2r\geq 0$.
If $n\geq 3r+2$ we have
$$\alpha=\left\lfloor \dfrac{n+1}{r+1} \right\rfloor\leq \dfrac{n+1}{r+1}\leq \dfrac{n-2}{r}$$
and then $r'=n-2-\alpha r\geq 0$. Furthermore, note that 
$$\alpha=\left\lfloor \dfrac{n+1}{r+1} \right\rfloor\leq \dfrac{n+1}{r+1}\leq \dfrac{n-3}{r-1}$$
and then $r''=n-3-\alpha (r-1)\geq 0$.

Now, we apply Proposition \ref{oscprojbir} with $l=\alpha, s_1=\dots = s_{\alpha-1}=r-1$ and $s_\alpha=r'$.
In this case the constraint in Proposition \ref{oscprojbir} is 
$\alpha\leq  n-r-1-(\alpha-1)(r-1)-r'$
that is $r'\leq n-2-\alpha r$.

Finally, if $l=\alpha, s_1=\dots = s_{\alpha-1}=r-2$ and $s_\alpha=r''$, then the constraint in Proposition \ref{oscprojbir} is
$\alpha\leq  n-r-1-(\alpha-1)(r-2)-r''$, that is $r''\leq  n-3-\alpha(r-1)$.
\end{proof}

\section{Degenerating tangential projections to osculating projections}\label{degtanosc}

In this section we construct explicit degenerations of tangential projections to osculating projections. We begin by studying how the span of two osculating spaces degenerates in a flat family of linear spaces parametrized by $\mathbb{P}^1$. 

We recall that the Grassmannian $\G(r,n)$ is rationally connected by rational normal curves of degree $r+1$. Indeed, if $p,q\in\G(r,n)$ are general points, corresponding to the $r$-planes $V_p,V_q\subseteq\mathbb{P}^n$, we may consider a rational normal scroll $X\subseteq\mathbb{P}^n$ of dimension $r+1$ containing $V_p$ and $V_q$. Then the $r$-planes of $X$ correspond to the points of a degree $r+1$ rational normal curve in $\G(r,n)$ joining $p$ and $q$.

The first step consists in studying how the span of two osculating spaces at two general points $p,q\in\G(r,n)$ behaves when $q$ approaches $p$ along a degree $r+1$ rational normal curve connecting $p$ and $q$.

\begin{Proposition}\label{limitosculatingspacesgrass}
Let $p,q\in\G(r,n)\subseteq\mathbb{P}^N$ be general points, $k_1,k_2\geq 0$ integers such that $k_1+k_2\leq r-1,$
and $\gamma:\P^1\to\G(r,n)$ a degree $r+1$ rational normal curve with $\gamma(0)=p$ and $\gamma(\infty)=q.$
Let us consider the family of linear spaces 
$$T_t=\left\langle T^{k_1}_p,T^{k_2}_{\gamma(t)}\right\rangle,\: t\in \P^1\backslash \{0\}$$
parametrized by $\P^1\backslash\{0\}$, and let $T_0$ be the flat limit of $\{T_t\}_{t\in \P^1\backslash \{0\}}$ in $\G(\dim(T_t),N)$. Then $T_0\subset T^{k_1+k_2+1}_p.$
\end{Proposition}
\begin{proof}
We may assume that $p=e_{I_1},q=e_{I_2}$, 
see  (\ref{I1Ialpha}),
and that $\gamma:\P^1\to\G(r,n)$ is the rational normal curve given by 
$$\gamma([t:s])= (se_0+te_{r+1})\wedge\dots\wedge (se_r+te_{2r+1}).$$
We can work on the affine chart $s=1$ and set $t=(t:1)$. Consider the points $$e_0,\dots,e_n,
e_0^{t}=e_0+te_{r+1},\dots,e_r^{t}=e_r+te_{2r+1},e_{r+1}^{t}=e_{r+1},\dots,e_n^{t}=e_n\in\P^n$$
and the corresponding points of $\P^N$ 
$$e_I=e_{i_0}\wedge\dots\wedge e_{i_r},e_I^{t}=e_{i_0}^{t}\wedge\dots\wedge e_{i_r}^{t},\: I\in \Lambda.$$
By Proposition \ref{oscgrass} we have 
$$T_t=\left\langle e_I\: | \: d(I,I_1)\leq k_1; \: e_I^{t}\: |\: d(I,I_1)\leq k_2\right\rangle
,\: t\neq 0$$
and
$$T^{k_1+k_2+1}_p=\left\langle e_I\: | \: d(I,I_1)\leq k_1+k_2+1\right\rangle
=\{p_I=0\: | \: d(I,I_1)> k_1+k_2+1\}.$$
Therefore, in order to prove that $T_0\subset T^{k_1+k_2+1}_p$ it is enough to exhibit, for any index $I\in \Lambda$ with $d(I,I_1)> k_1+k_2+1$, a hyperplane $H_I\subset\mathbb{P}^N$ of type
$$p_I+t\left( \sum_{J\in \Lambda, \ J\neq I}f(t)_{I,J} p_J\right)=0$$
such that $T_t\subset H_I$ for $t\neq 0$, where $f(t)_{I,J}\in \C[t]$ are polynomials. Clearly, taking the limit for $t\mapsto 0$, this will imply that $T_0\subseteq \{p_I = 0\}$.

In order to construct such a hyperplane we need to introduce some other definitions. We define
$$\Delta(I,l):= \left\{(I\setminus J)\cup (J+r+1)|\: J\subset I\cap I_1,\: |J|=l\right\}
\subset \Lambda$$
for any $I\in \Lambda,\: l\geq 0$, where $L+\lambda:=\{i+\lambda;\: i\in L\}$ is the translation of the set $L$ by the integer $\lambda$. Note that $\Delta(I,0)=\{I\}$ and $\Delta(I,l)=\emptyset$ for $l$ big enough. For any $l>0$ set
$$\Delta(I,-l):=\left\{ J|\: I\in\Delta(J,l) \right\} \subset \Lambda;$$
$$s_I^+:=\max_{l\geq 0}\{\Delta(I,l)\neq \emptyset\}\in\{0,\dots,r+1\};$$
$$s_I^-:=\max_{l\geq 0}\{\Delta(I,-l)\neq \emptyset\}\in\{0,\dots,r+1\};$$
$$\Delta(I)^+:=\bigcup_{0\leq l} \Delta(I,l)=
\!\!\!\!\bigcup_{0\leq l \leq s_I^+} \!\!\!\! \Delta(I,l);$$
$$\Delta(I)^-:=\bigcup_{0\leq l} \Delta(I,-l)=
\!\!\!\!\bigcup_{0\leq l \leq s_I^-} \!\!\!\!\Delta(I,-l).$$
Note that $0\leq s_I^-\leq d(I,I_1),0\leq s_I^+\leq r+1-d(I,I_1)$, and for any $l$ we have
$$J\in\Delta(I,l)\Rightarrow d(J,I)=|l|, d(J,I_1)=d(I,I_1)+l,d(J,I_2)=d(I,I_2)-l.$$
In order to get acquainted with the rest of the proof the reader may keep reading the proof taking a look to Example \ref{exampledeg} where we follow the same lines of the proof in the case $(r,n) = (2,5)$.

Now, we write the $e_I^t$'s with $d(I,I_1)< k_2,$ in the basis $e_J,J\in\Lambda$. For any $I\in\Lambda$ we have
\begin{align*}
e_{I}^{t}
&=e_I+t\!\!\!\!\sum_{J\in \Delta(I,1)}\!\!\!\!\left(\sign(J){e_J}\right)+\dots
+t^{l}\!\!\!\!\!\sum_{J\in \Delta(I,l)}\!\!\!\!\left(\sign(J){e_J}\right)+\dots
+t^{s_I^+}\!\!\!\!\!\!\!\!\sum_{J\in \Delta(I,s_I^+)}\!\!\!\!\!\left(\sign(J){e_J}\right)\\
&=\sum_{l=0}^{s_I^+}\left(t^l\!\!\!\!\sum_{J\in \Delta(I,l)}\sign(J) e_J\right)
=\!\!\!\!\sum_{J\in \Delta(I)^+}\!\!\!\!\left( t^{d(I,J)}\sign(J)e_J\right)
\end{align*}
where $\sign(J)=\pm 1$. Note that $\sign(J)$ depends on $J$ but not on $I$, hence we may replace $e_J$ by $\sign(J)e_J$, and write 
\begin{align*}
e_{I}^{t}=\sum_{J\in \Delta(I)^+}\!\!\!\!t^{d(I,J)}e_J.
\end{align*}
Therefore, we have
\begin{align*}
T_t=&\Big< e_I \: | \: d(I,I_1)\leq k_1;\:
\sum_{J\in \Delta(I)^+}\!\!\!\!\left( t^{d(I,J)}e_J\right) \: | \: d(I,I_1)\leq k_2
\Big >.
\end{align*}
Next, we define $$ \Delta:=\left\{ I \: | \:  d(I,I_1)\leq k_1\right\}\bigcup
\left(\bigcup_{d(I,I_1)\leq k_2} \!\!\!\!\Delta(I)^+\right)\subset \Lambda.$$
Let $I\in \Lambda$ be an index with $d(I,I_1)> k_1+k_2+1$. If $I\notin \Delta$ then $T_t\subset \{p_I=0\}$ for any $t\neq 0$ and we are done.

Now, assume that $I\in \Delta$. For any $e_K^t$ with non-zero Pl\"ucker coordinate $p_I$ we have $I\in \Delta(K)^+$, that is $K\in \Delta(I)^-$. Now, we want to find a hyperplane $H_I$ of type
\begin{align}\label{hyperplane}
F_I=\sum_{J\in \Delta(I)^-}t^{d(I,J)}c_J p_J =0
\end{align}
where $c_J\in\C$ with $c_I\neq 0$, and such that $T_t\subset H_I$ for $t\neq 0$. Note that then we can divide the equation by $c_I$, and get a hyperplane $H_I$ of the required type:
$$p_I+\frac{t}{c_I}\left(\sum_{J\in \Delta(I)^-, \ J\neq I}t^{d(J,I)-1} c_J p_J\right)=0$$
In the following we will write $s_I^-=s$ for short. Since 
$$\left|\Delta(I)^-\right|
=\sum_{l=0}^{s}\left|\Delta(I,-l)\right|
=1+s+\binom{s}{2}+\dots+\binom{s}{s -1}+1=2^{s}$$
in equation (\ref{hyperplane}) there are $2^s$ variables $c_J$. Now, we want to understand what conditions we get by requiring $T_t\subseteq \{F_I=0\}$ for $t\neq 0$.

Given $K\in \Delta(I)^-$ we have $s_K^+\geq d(I,K)$ and
\begin{align*}
F_I(e_K^t)&=
F\left(\sum_{L\in \Delta(K)^+}\!\!\!\!\left( t^{d(K,L)}e_L\right)\right)
=F\left(\sum_{l=0}^{s_K^+}\left(t^l\!\!\!\!\sum_{L\in \Delta(K,l)} e_L\right)\right)
=F\left(\sum_{l=0}^{d(I,K)}\left(t^l\!\!\!\!\sum_{L\in \Delta(K,l)} e_L\right)\right)\\
&\stackrel{(\ref{hyperplane})}{=}\sum_{J\in \Delta(I)^-\cap \Delta(K)^+}t^{d(I,K)-d(J,K)}c_J 
\left(  t^{d(J,K)}\right)
=t^{d(I,K)}\left[\sum_{J\in \Delta(I)^-\cap \Delta(K)^+}c_J \right]
\end{align*}
that is 
\begin{align*}
F_I(e_K^t)=0\ \forall t\neq 0 \Leftrightarrow \sum_{J\in \Delta(I)^-\cap \Delta(K)^+}\!\!\!\!\!\!c_J =0.
\end{align*}
Note that this is a linear condition on the coefficients $c_J$, with $J\in \Delta(I)^-$. Therefore,
\begin{align}\label{eqns3}
T_t\subset \{F_I=0\} \mbox{ for } t\neq 0 &\Leftrightarrow 
\begin{cases}
F_I(e_L)=0\   &\forall L\in \Delta(I)^-\cap B[I_1,k_1]  \\ 
F_I(e_K^t)=0 \ \forall t\neq 0 \  &\forall K\in \Delta(I)^-\cap B[I_1,k_2] 
\end{cases}\\&\nonumber
\Leftrightarrow 
\begin{cases}
c_L=0  &\forall L\in \Delta(I)^-\cap B[I_1,k_1]\\
\displaystyle\sum_{J\in \Delta(I)^-\cap \Delta(K)^+}\!\!\!\!\!\!\!\!\!\!\!\!c_J \quad
=0 \ &\forall K\in \Delta(I)^-\cap B[I_1,k_2]
\end{cases}
\end{align}
where $B[J,u]:=\{K\in \Lambda |\: d(J,K)\leq u\}$. The number of conditions on the $c_J$'s, $J\in \Delta(I)^-$ is then
$$c:=\left|\Delta(I)^-\cap B[I_1,k_1]\right|+\left|\Delta(I)^-\cap B[I_1,k_2]\right|.$$

The problem is now reduced to find a solution of the linear system given by the $c$ equations (\ref{eqns3}) in the $2^s$ variables $c_J$'s, $J\in \Delta(I)^-$ such that $c_I\neq 0$. Therefore, it is enough to find $s+1$ complex numbers $c_I=c_0\neq 0,c_1,\dots, c_{s}$ satisfying the following conditions 
\begin{align}\label{eqns4}
\begin{cases}
c_{j}=0  &\forall j=s,\dots,d-k_1\\
\displaystyle\sum_{l=0}^{d(I,K)}\left|\Delta(I)^-\cap \Delta(K,l)\right|c_{d(I,K)-l} =0 
\ &\forall K\in \Delta(I)^-\cap B[I_1,k_2]
\end{cases}
\end{align}
where $d=d(I,I_1)>k_1+k_2+1$. Note that (\ref{eqns4}) can be written as
\begin{align*}
\begin{cases}
c_{j}=0  &\forall j=s,\dots,d-k_1\\
\displaystyle\sum_{k=0}^{j}\binom{j}{j-k}c_{k} =0 
\ &\forall j=s,\dots,d-k_2
\end{cases}
\end{align*}
that is
\begin{align}\label{eqns5}
\begin{cases}
c_{s}=0\\
\vdots\\
c_{d-k_1}=0\\
\end{cases}
\begin{cases}
\binom{s}{0}c_{s}+\binom{s}{1}c_{s-1}+\cdots+\binom{s}{s-1}c_{1}+\binom{s}{s}c_{0}=0\\
\vdots\\
\binom{d-k_2}{0}c_{d-k_2}+\binom{d-k_2}{1}c_{d-k_2-1}+\cdots+\binom{d-k_2}{d-k_2-1}c_{1}+\binom{d-k_2}{d-k_2}c_{0}=0
\end{cases}
\end{align}
Now, it is enough to show that the linear system (\ref{eqns5}) admits a solution with $c_0\neq 0$. If $s<d-k_2,$ the system (\ref{eqns5}) reduces to $c_s=\dots=c_{d-k_1}=0$. In this case we may take $c_0=1, c_1=\dots,c_s=0$. Note that $d-k_1>k_2+1\geq 1$ and we can use the hyperplane $p_I=0$.

From now on assume that $s\geq d-k_2$. Since $c_s=\dots=c_{d-k_1}=0$ we may consider the second set of conditions in (\ref{eqns5}) and translate the problem into checking that the system (\ref{eqns6}) admits a solution involving the variables $c_0,c_1,\dots,c_{d-k_1+1}$ with $c_0\neq 0$. Note that (\ref{eqns5}) takes the following form:
\begin{align}\label{eqns6}
\begin{cases}
\binom{s}{s-(d-k_1+1)}c_{d-k_1+1}+\binom{s}{s-(d-k_1)}c_{d-k_1}+
\cdots+\binom{s}{s-1}c_{1}+\binom{s}{s}c_{0}=0\\
\vdots\\
\binom{d-k_2}{k_1-1-k_2}c_{d-k_1+1}+\binom{d-k_2}{k_1-k_2}c_{d-k_1}+
\cdots+\binom{d-k_2}{d-k_2-1}c_{1}+\binom{d-k_2}{d-k_2}c_{0}=0
\end{cases}
\end{align}
Therefore, it is enough to check that the $(s-d+k_2+1)\times (d-k_1+1)$ matrix
\begin{align}\label{eqns7}
M=
\begin{pmatrix}
\binom{s}{s-(d-k_1+1)}&\binom{s}{s-(d-k_1)}&
\cdots&\binom{s}{s-1}\\
\vdots&\vdots& &\vdots\\
\binom{d-k_2}{k_1-1-k_2}&\binom{d-k_2}{k_1-k_2}&
\cdots&\binom{d-k_2}{d-k_2-1}
\end{pmatrix}
\end{align}
has maximal rank. Note that $s\leq d$ and $d>k_1+k_2+1$ yield $s-d+k_2+1< d-k_1+1$. Then it is enough to show that the $(s-d+k_2+1)\times (s-d+k_2+1)$ submatrix
\begin{align*}
M'=&
\begin{pmatrix}
\binom{s}{s-(s-d+k_2+1)}&\binom{s}{s-(s-d+k_2)}&
\cdots&\binom{s}{s-1}\\
\vdots&\vdots& &\vdots\\
\binom{d-k_2}{d-k_2-(s-d+k_2+1)}&\binom{d-k_2}{d-k_2-(s-d+k_2)}&
\cdots&\binom{d-k_2}{d-k_2-1}
\end{pmatrix}\\
=&
\begin{pmatrix}
\binom{s}{d-k_2-1}&\binom{s}{d-k_2}&
\cdots&\binom{s}{s-1}\\
\vdots&\vdots& &\vdots\\
\binom{d-k_2}{2d-2k_2-s-1}&\binom{d-k_2}{2d-2k_2-s}&
\cdots&\binom{d-k_2}{d-k_2-1}
\end{pmatrix}=
\begin{pmatrix}
\binom{s}{s+1-d+k_2}&\binom{s}{s-d+k_2}&
\cdots&\binom{s}{1}\\
\vdots&\vdots& &\vdots\\
\binom{d-k_2}{s+1-d+k_2}&\binom{d-k_2}{s-d+k_2}&
\cdots&\binom{d-k_2}{1}
\end{pmatrix}
\end{align*}
has non-zero determinant. Since the determinant of $M'$ is equal to the determinant of the matrix of binomial coefficients 
$$M'':=\left(\binom{i}{j}\right)_{\substack{ \hspace{-0.7cm}d-k_2\leq i\leq s\\ 1\leq j\leq s+1-d+k_2}}.$$
it is enough to observe that since $d-k_2>k_1+1\geq 1$ by \cite[Corollary 2]{GV85} we have $\det(M')=\det(M'')\neq 0$.
\end{proof}
In the following example we work out explicitly the proof of Proposition \ref{limitosculatingspacesgrass}.

\begin{Example}\label{exampledeg}
Consider the case $(r,n)=(2,5).$ Then $I_1=\{0,1,2\}, I_2=\{3,4,5\}$. Let us take 
$$I^1=\{0,1,2\},I^2=\{0,1,3\},I^3=\{0,4,5\}.$$
Then we have 
$$
\begin{array}{ll}
\Delta(I^1,1)=\{\{1,2,3\},\{0,2,4\},\{0,1,5\}\} & \Delta(I^3,1)=\{\{3,4,5\}\}\\
\Delta(I^1,2)=\{\{0,4,5\},\{1,3,5\},\{2,3,4\}\} & \Delta(I^3,-1)=\{\{0,1,5\},\{0,2,4\}\}\\
\Delta(I^1,3)=\{\{3,4,5\}\} &\Delta(I^3,-2)=\{\{0,1,2\}\} \\
\Delta(I^2,1)=\{\{0,3,4\}\} &
\end{array}
$$
and $\Delta(I^j,l)=\emptyset$ for any other pair $(j,l)$ with $1\leq j\leq 3$ and $l\neq 0$. Therefore 
$$s^+_{I^1}=3,s^-_{I^1}=0,s^+_{I^2}=1,s^-_{I^2}=0,s^+_{I^3}=1,s^-_{I^3}=2$$
while
$$d(I^1,I_1)=0,d(I^2,I_1)=1,d(I^3,I_1)=2.$$
Let us work out the case $k_1=0,k_2=1$. Here $T_p^{k_1}$ is just the point $e_{012}$ and the generators of $T^{k_2}_{\gamma(t)}$ are
$$e_{012}^t,e_{123}^t,e_{024}^t,e_{015}^t,
e_{124}^t,e_{125}^t,
e_{023}^t,e_{025}^t,
e_{013}^t,e_{014}^t
$$
We can write them on the basis $(e_I)_{I\in \Lambda}$ as
\begin{align}\label{etonthebasisexampleI}
\begin{cases}
e_{012}^t&=e_{012}+t(e_{123}+e_{024}+e_{015})+t^2(e_{045}+e_{135}+e_{234})+t^3e_{345}\\
e_{123}^t&=e_{123}+t(e_{135}+e_{234})+t^2e_{345}\\
e_{024}^t&=e_{024}+t(e_{045}+e_{234})+t^2e_{345}\\
e_{015}^t&=e_{015}+t(e_{045}+e_{135})+t^2e_{345}\\
\end{cases}
\end{align}
and
\begin{align}\label{etonthebasisexampleII}
\begin{cases}
e_{124}^t&=e_{124}+te_{145}\\
e_{125}^t&=e_{125}+te_{245}\\
e_{023}^t&=e_{023}+te_{035}\\ 
e_{025}^t&=e_{025}+te_{235}\\
e_{013}^t&=e_{013}+te_{034}\\ 
e_{014}^t&=e_{014}+te_{134}
\end{cases}
\end{align}
Now, given $I\in \Lambda$ with $d(I,I_1)>2=k_1+k_2+1$ we have to find a hyperplane $H_I$ of type
$$c_Ip_I+t \!\!\!\!\sum_{J\in \Delta(I,-1)}\!\!\!\! c_J p_J+
t^2 \!\!\!\!\sum_{J\in \Delta(I,-2)}\!\!\!\! c_J p_J+
t^3 \!\!\!\!\sum_{J\in \Delta(I,-3)}\!\!\!\! c_J p_J=0$$
such that $c_I\neq 0$, and $T_t\subseteq H_I$ for every $t\neq 0.$

In this case it is enough to consider $I=\{3,4,5\}$. Note that $e_{345}$ appears in (\ref{etonthebasisexampleI}) but does not in (\ref{etonthebasisexampleII}). In the notation of the proof of Proposition \ref{limitosculatingspacesgrass} we have $d=s=3, d-k_1=3,d-k_2=2$, and we are looking for
$$c_0=c_{345}\neq 0, c_1=c_{045}=c_{135}=c_{234}, c_2=c_{123}=c_{024}=c_{015},c_3=c_{012}$$
satisfying the following system:
\begin{align}\label{eqnsexample}
\begin{cases}
c_{3}=0\\
\binom{3}{0}c_{3}+\binom{3}{1}c_{2}+\binom{3}{2}c_{1}+\binom{3}{3}c_{0}=0\\
\binom{2}{0}c_{2}+\binom{2}{1}c_{1}+\binom{2}{2}c_{0}=0
\end{cases}
\end{align}
Note that the matrix
\begin{equation*}
M=
\begin{pmatrix}
\binom{3}{1}&\binom{3}{2}\\
\binom{2}{0}&\binom{2}{1}
\end{pmatrix}=
\begin{pmatrix}
3&3\\
1&2
\end{pmatrix}
\end{equation*}
has maximal rank. Therefore, there exist complex numbers $c_I=c_0\neq 0,c_1,c_2,c_3$ satisfying system (\ref{eqnsexample}). For instance, we may take $c_0=3,c_1=-2,c_2=1,c_3=0$ corresponding to the hyperplane  
$$3p_{345}-2t(p_{045}+p_{135}+p_{234})+t^2(p_{123}+p_{024}+p_{015})=0$$
and taking the limit for $t\mapsto 0$ we get the equation $p_{345}=0$.
\end{Example}

Essentially, Proposition \ref{limitosculatingspacesgrass} says that two general osculating spaces of $\G(r,n)$ behave well under degenerations. We formalize this concept as follows.

\begin{Assumption}\label{ass}
Let $X\subset \P^N$ be an irreducible projective variety, $p,q\in X$ be general points, and $k_1,k_2\geq 0$ integers. We will assume that there exists a smooth curve $\gamma:C\to X$, with $\gamma(t_0)=p$ and $\gamma(t_\infty)=q$ such that the flat limit $T_{t_0}$ in $\G(dim(T_t),N)$ of the family of liner spaces 
$$T_t=\left\langle T^{k_1}_p,T^{k_2}_{\gamma(t)}\right\rangle,\: t\in C\backslash \{t_0\}$$
parametrized by $C\backslash \{t_0\}$, is contained in $T^{k_1+k_2+1}_p$.
\end{Assumption}

For our applications to Grassmannians we will always choose $C\cong \mathbb{P}^1$. Moreover, we would like to stress that there exist varieties, with small higher order osculating spaces, not satisfying Assumption \ref{ass}. 

\begin{Example}
Let us consider the tangent developable $Y_n\subseteq\mathbb{P}^n$ of a degree $n$ rational normal curve $C_n\subseteq\mathbb{P}^n$ as in Proposition \ref{tdrnc}.

Note that two general points $p = \phi(t_1,u_1)$, $q = \phi(t_2,u_2)$ in $Y_n$ can be joined by a smooth rational curve. Indeed, we may consider the curve 
$$\xi(t) = (t_1+t(t_2-t_1)+u_1+t(u_2-u_1),\dots, (t_1+t(t_2-t_1))^n+n(t_1+t(t_2-t_1))^{n-1}(u_1+t(u_2-u_1)))$$ 
Now, let $\gamma:C\rightarrow Y_n$ be a smooth curve with $\gamma(t_0)=p$ and $\gamma(t_\infty)=q$, and let $T_{t_0}$ be the flat limit of the family of liner spaces 
$$T_t=\left\langle T_{p},T_{\gamma(t)}\right\rangle,\: t\in C\backslash \{t_0\}.$$
Now, one can prove that if $n\geq 5$ then $T_{p}Y_n\cap T_{q}Y_n = \emptyset$ by a straightforward computation, or alternatively by noticing that by \cite{Ba05} $Y_n$ is not $2$-secant defective, and then by Terracini's lemma \cite[Theorem 1.3.1]{Ru03} $T_{p}Y_n\cap T_{q}Y_n = \emptyset$. Now, $T_{p}Y_n\cap T_{q}Y_n = \emptyset$ implies that $\dim(T_t) = 5$ for any $t\in C$. On the other hand, by Proposition \ref{tdrnc} we have $\dim(T^3_pY_n) = 4$. Hence, $T_{t_0}\nsubseteq T^3_pY_n$ as soon as $n\geq 5$.
\end{Example}

Now we are ready to prove a stronger version of Proposition \ref{limitosculatingspacesgrass}.

\begin{Proposition}\label{limitosculatingspacesgrassII}
Let $p_1,\dots,p_{\alpha}\in \G(r,n)\subseteq\mathbb{P}^{N}$ be general points with $\alpha = \lfloor\frac{n+1}{r+1}\rfloor$,
$k\leq (r-1)/2$ a non-negative integer, and $\gamma_j:\P^1\to \G(r,n)$
a degree $r+1$ rational normal curve with $\gamma_j(0)=p_1$ and $\gamma_j(\infty)=p_j,$ for every $j=2,\dots,\alpha.$
Let us consider the family of linear spaces 
$$T_t=\left\langle T^{k}_{p_1},T^{k}_{\gamma_2(t)},\dots,T^{k}_{\gamma_{\alpha}(t)}
\right\rangle,\: t\in \P^1\backslash \{0\}$$
parametrized by $\P^1\backslash\{0\}$, and let $T_0$ be the flat limit of $\{T_t\}_{t\in \P^1\backslash \{0\}}$ in
$\G(\dim(T_t),N)$. Then $T_0\subset T^{2k+1}_p.$
\end{Proposition}
\begin{proof}
If $\alpha=2$ it follows from the Proposition \ref{limitosculatingspacesgrass}. Therefore, we may assume that $\alpha\geq 3$, $p_j=e_{I_j}$ (\ref{I1Ialpha}) and that $\gamma_j:\P^1\to\mathbb{P}^{N}$ is the rational curve given by 
$$\gamma_j([t:s])=\left(se_0+te_{(r+1)(j-1)}\right)\wedge\dots \wedge \left(se_{r}+te_{(r+1)j-1}\right).$$
We can work on the affine chart $s=1$ and set $t=(t:1)$. Consider the points 
$$e_0,\dots,e_{n},
e_0^{j,t}=e_0+te_{(r+1)(j-1)},
\dots,e_{r}^{j,t}=e_{r}+te_{(r+1)j-1},e_{r+1}^{j,t}=e_{r+1},\dots,e_{n}^{j,t}=e_{n}\in \P^{n}$$
and the corresponding points in $\mathbb{P}^{N}$ 
$$e_I=e_{i_0}\wedge\dots\wedge e_{i_{r}},e_I^{j,t}=e_{i_0}^{j,t}\wedge\dots\wedge e_{i_{r}}^{j,t},\:
I=\{i_0,\dots,i_r\}\in \Lambda,$$
for $j=2,\dots,\alpha.$
By Proposition \ref{oscgrass} we have 
$$T_t=\left\langle e_I\: | \: d(I,I_1)\leq k; \: e_I^{j,t}\: |\: d(I,I_1)\leq k, j=2,\dots,\alpha\right\rangle
,\: t\neq 0$$
and
$$T^{2k+1}_{p_0}=\left\langle e_I\: | \: d(I,I_1)\leq 2k+1\right\rangle
=\{p_I=0\: | \: d(I,I_1)> 2k+1 \}.$$
Therefore, as in Proposition \ref{limitosculatingspacesgrass}, in order to prove that $T_0\subset T^{2k+1}_p$ it is enough to exhibit, for any index $I\in \Lambda$ with $d(I,I_1)> 2k+1$, a hyperplane $H_I\subset\mathbb{P}^{N}$ of type
$$p_I+t\left( \sum_{J\in \Lambda, \ J\neq I}f(t)_{I,J} p_J\right)=0$$
such that $T_t\subset H_I$ for $t\neq 0$, where $f(t)_{I,J}\in \C[t]$ are polynomials. The first part of the proof goes as in the proof of Proposition \ref{limitosculatingspacesgrass}. Given $I\in \Lambda$ we define 
$$\Delta(I,l)_j:=\left\{(I\setminus J)\cup(J+(j-1)(r+1))| J\subset I\cap I_1, |J|=l\right\}\subset \Lambda$$
for any $I\in \Lambda,l\geq 0, j=2,\dots, \alpha,$ where $L+\lambda:=\{i+\lambda; i\in L\}$ is the 
translation of the set $L$ by the integer $\lambda.$ Note that $\Delta(I,0)_j=\{I\}$ and $\Delta(I,l)_j=\emptyset$ for $l$ big enough. For any $l>0$ set
$$\Delta(I,-l)_j:=\left\{ J|\: I\in\Delta(J,l)_j \right\} \subset \Lambda;$$
$$s(I)^+_j:=\max_{l\geq 0}\{\Delta(I,l)_j\neq \emptyset\}\in\{0,\dots,r+1\};$$
$$s(I)^-_j:=\max_{l\geq 0}\{\Delta(I,-l)_j\neq \emptyset\}\in\{0,\dots,r+1\};$$
$$\Delta(I)^+_j:=\bigcup_{0\leq l} \Delta(I,l)_j=
\!\!\!\!\bigcup_{0\leq l \leq s(I)^+_j} \!\!\!\! \Delta(I,l)_j;$$
$$\Delta(I)^-_j:=\bigcup_{0\leq l} \Delta(I,-l)_j=
\!\!\!\!\bigcup_{0\leq l \leq s(I)^-_j} \!\!\!\!\Delta(I,-l)_j.$$
Note that $0\leq s(I)^-_j\leq d(I,I_1),0\leq s(I)^+_j\leq r+1-d(I,I_1)$, and for any $l$ we have
$$J\in\Delta(I,l)_j\Rightarrow d(J,I)=|l|, d(J,I_1)=d(I,I_1)+l,d(J,I_j)=d(I,I_j)-l.$$
 Now, we write $e_I^{j,t}$, $d(I,I_1)< k$, in the basis $e_J,J\in\Lambda$. For any $I\in\Lambda$ we have
\begin{align*}
e_{I}^{j,t}=\!\!\!\!\sum_{J\in \Delta(I)^+_j}\!\!\!\!\left( t^{d(I,J)}\sign(J)e_J\right)
\end{align*}
where $\sign(J)=\pm 1$. Since $\sign(J)$ does depend on $J$ but not on $I$ we can replace $e_J$ by $\sign(J)e_J$. Then, we may write 
\begin{align*}
e_{I}^{t}=\sum_{J\in \Delta(I)^+_j}\!\!\!\!\left( t^{d(I,J)}e_J\right).
\end{align*}
and
\begin{align*}
T_t=&\left\langle e_I \: | \: d(I,I_1)\leq k;\:
\sum_{J\in \Delta(I)^+_j}\!\!\!\!\left( t^{d(I,J)}e_J\right) \: \big| \: d(I,I_1)\leq k,\: 2\leq j\leq \alpha\right \rangle.
\end{align*}
Next, we define 
$$ \Delta:=\left\{ I \: | \:  d(I,I_1)\leq k\right\}\bigcup
\left(\bigcup_{2\leq j\leq \alpha}\bigcup_{d(I,I_1)\leq k} \!\!\!\!\Delta(I)^+_j\right)\subset \Lambda.$$
Let $I\in \Lambda$ be an index with $d(I,I_1)=:D> 2k+1$. If $I\notin \Delta$ then $T_t\subset \{p_I=0\}$ for any $t\neq 0$ and we are done. Now, assume that $I\in \Delta$, and $I\in \Delta(K_1,l_1)_2 \bigcap \Delta(K_2,l_2)_3$ with 
$$d(K_1,I_1),d(K_2,I_1)\leq k.$$
Consider the following sets
\begin{align*}
I^0:&=I\cap I_1\\
I^1:&=I\cap(K_1+(r+1))\subset I_2\\
I^2:&=I\cap(K_2+2(r+1))\subset I_3\\
I^3:&=I\setminus (I^0\cup I^1 \cup I^2)
\end{align*}
Then $|I^1|=l_1,|I^2|=l_2$. Set $u:=|I^3|$, then
$$d(I,I_1)=l_1+l_2+u\leq l_1+l_2+2u= d(K_1,I_1)+d(K_2,I_1)\leq 2k$$
contradicting $d(I,I_1)>2k+1$. Therefore, there is a unique $j$ such that 
$$I\in \bigcup_{d(J,I_1)\leq k} \!\!\!\!\Delta(J)^+_j.$$
Note that $\Delta(I,-s(I)^-_j)$ has only one element, say $I'$. Then 
$$k+1-D+s(I)_j^-=k+1-d(I,I_1)+d(I,I')=k+1-d(I_1,I')>0.$$
Now, consider the set of indexes
$$\Gamma:=\left\{ I \right\} \cup \Delta(I,-1)_j \cup \dots \cup \Delta(I,-(k+1-D+s(I)^-_j))_j
=\!\!\!\!\!\!\!\!\bigcup_{0\leq l \leq k+1-D+s(I)^-_j} \!\!\!\!\!\!\!\!\Delta(I,-l)_j \subset \Lambda$$
Our aim now is to find a hyperplane of the form
\begin{equation}\label{eq1}
H_I= \left \{\sum_{J\in \Gamma } t^{d(I,J)}c_J p_J=0 \right\}
\end{equation}
such that $T_t\subset H_I$ and $c_I\neq 0$.

First, we claim that 
\begin{equation}\label{eq2}
J\in \Gamma\Rightarrow J \notin  \bigcup_{\substack{2\leq i\leq \alpha \\ i\neq j} }
\bigcup_{d(I,I_1)\leq k} \!\!\!\!\Delta(I)^+_i.
\end{equation}
Indeed, assume that $J\in \Delta(I,-l)_j\cap\Delta(K,m)_i$ for some $K\in \Lambda$ with 
$$d(K,I_1)\leq k \mbox{ and } i\neq j, 0\leq l \leq k+1-D+s(I)^-_j, m\geq 0.$$
Since $J\in \Delta(I,-l)_j$ then
\begin{equation*}
|J\cap I_j|=|I\cap I_j|-l\geq s(I)_j^- -l\geq D-(k+1)>k
\end{equation*}
On other hand, since $J\in \Delta(K,m)_i$ with $j\neq i$ we have
\begin{equation*}
|J\cap I_j|=|K\cap I_j|\leq d(K,I_1)\leq k.
\end{equation*}
A contradiction. Now, (\ref{eq2}) yields that the hyperplane $H_I$ given by (\ref{eq1}) is such that 
$$\left\langle e_I\: | \: d(I,I_1)\leq k; \: \sum_{J\in \Delta(I)^+_i}\!\!\!\!
t^{d(I,J)} c_{(I,J)}e_J\: | \: d(I,I_1)\leq k,\ i=2,\dots,\alpha, i\neq j \right\rangle \subset H_I , t\neq 0.$$
Therefore
$$
T_t\subset H_I, t \neq 0 \Longleftrightarrow \left\langle \sum_{J\in \Delta(I)^+_j}\!\!\!\!
t^{d(I,J)} e_J\: | \: d(I,I_1)\leq k \right\rangle \subset H_I , t\neq 0
.$$
Now, arguing as in the proof of Proposition \ref{limitosculatingspacesgrass} we obtain
\begin{equation}\label{eq4II}
T_t\subset H_I, t \neq 0 \Longleftrightarrow \sum_{J\in \Delta(K)^+_j \cap \Gamma}\!\!\!\! c_{J}=0
\ \ \forall  K\in \Delta(I)^-_j\cap B[I_1,k]
\end{equation}
and the problem is now reduced to find a solution of the linear system given by the 
$|\Delta(I)^-_j\cap B[I_1,k]|$ equations (\ref{eq4II}) in the $|\Delta(K)^+_j \cap \Gamma|$ variables $c_J$, $J\in\Delta(K)^+_j \cap \Gamma$, such that $c_I\neq 0$. We set $c_J=c_{d(I,J)}$ and, as in the proof of Proposition \ref{limitosculatingspacesgrass}, we consider the linear system 
\begin{equation}\label{linsys}
\displaystyle\sum_{l=0}^{k+1-D+s(I)^-_j}\binom{D-i}{D-l-i}c_{l} =0 
\quad \forall i=D-s(I)^-_j,\dots,k
\end{equation}
with $k+2-D+s(I)^-_j$ variables $c_0,\dots,c_{k+1-D+s(I)^-_j}$ and $k+1-D+s(I)^-_j$ equations, where $D=d(I,I_1)$. Finally, arguing exactly as in the last part of the proof of Proposition \ref{limitosculatingspacesgrass} we have that (\ref{linsys}) admits a solution with $c_0\neq 0$.
\end{proof}

We conclude this section with the definition of \textit{$m$-osculating regularity} which essentially will be a measure of how many general osculating spaces of order $k$ we can degenerate to an osculating space of order $2k+1$.

\begin{Definition}\label{moscularity}
Let $X\subset\mathbb{P}^N$ be an irreducible projective variety. We say that $X$ has \textit{$m$-osculating regularity} if given $p_1,\dots,p_{m}\in X$ general points, and an integer $k\geq 0$, there exist smooth curves $\gamma_j:C\to X$
with $\gamma_j(t_0)=p_1$ and $\gamma_j(t_\infty)=p_j$ for $j=2,\dots,m$ such that the family of linear spaces 
$$T_t=\left\langle T^{k}_{p_1},T^{k}_{\gamma_2(t)},\dots,T^{k}_{\gamma_{m}(t)}\right\rangle,\: t\in C\backslash \{t_0\}$$
parametrized by $C\backslash\{t_0\}$ has flat limit $T_{t_0}$ contained in $T^{2k+1}_p$.
\end{Definition}

Note that by Proposition \ref{limitosculatingspacesgrassII} the Grassmannian $\G(r,n)$ has $\alpha$-osculating regularity, where $\alpha = \lfloor \frac{n+1}{r+1}\rfloor$.

\subsection{Limit linear systems}\label{lls}
Let $X\subset\mathbb{P}^N$ be an irreducible rational variety of dimension $n$, $p_1,\dots,p_m\in X$ general points. We reinterpret the notion of $m$-osculating regularity in Definition \ref{moscularity} in terms of limit linear systems and collisions of fat points. 

Let $\mathcal{H}\subseteq |\mathcal{O}_{\mathbb{P}^n}(d)|$ be the sublinear system of $|\mathcal{O}_{\mathbb{P}^n}(d)|$ inducing the birational map $i_{\mathcal{H}}:\mathbb{P}^n\dasharrow X\subset\mathbb{P}^N$, and $q_i = i_{\mathcal{H}}^{-1}(p_i)$. 

Then $X$ has $m$-osculating regularity if and only if there exists smooth curves $\gamma_i:C\rightarrow \mathbb{P}^n$, $i = 2,\dots,m$, with $\gamma_i(t_0) = q_1$, $\gamma_i(t_{\infty}) = q_i$ for $i = 1,\dots,m$, such that the limit linear system $\mathcal{H}_{t_0}$ of the family of linear systems $\mathcal{H}_t$ given by the hypersurfaces in $\mathcal{H}$ having at least multiplicity $s+1$ at $q_1,\gamma_2(t),\dots,\gamma_m(t)$ contains the linear system $\mathcal{H}^{2s+2}_{q_1}$ of degree $d$ hypersurfaces with multiplicity at least $2s+2$ at $q_1$.

Indeed, if $p_i = i_{\mathcal{H}}(q_i)$ for $i = 1,\dots,m$ then the linear system of hyperplanes in $\mathbb{P}^N$ containing 
$$T_t = \left\langle T_{p_1}^s,T_{i_{\mathcal{H}}(\gamma_2(t))}^s,\dots, T_{i_{\mathcal{H}}(\gamma_m(t))}^s\right\rangle$$
corresponds to the linear system $\mathcal{H}_t$. Similarly, the linear system of hyperplanes in $\mathbb{P}^N$ containing $T_{p_1}^{2s+1}$ corresponds to the linear system $\mathcal{H}^{2s+2}_{q_1}$.

Therefore, the problem of computing the $m$-osculating regularity of a rational variety can be translated in terms of limit linear systems in $\mathbb{P}^n$ given by colliding a number of fat points. This is a very hard and widely studied subject \cite{CM98}, \cite{CM00}, \cite{CM05}, \cite{Ne09}.

\subsection{Degenerating rational maps}
In order to study the fibers of general tangential projections via osculating projections we need to understand how the fibers of rational maps behave under specialization. We refer to \cite{GD64} for the general theory of rational maps relative to a base scheme.

\begin{Proposition}\label{p1}
Let $C$ be a smooth and irreducible curve, $X\rightarrow C$ an integral scheme flat over $C$, and $\phi:X\dasharrow \mathbb{P}^n_{C}$ be a rational map of schemes over $C$. Let $d_0 = \dim(\overline{\phi_{|X_{t_0}}(X_{t_0})})$ with $t_0\in C$. Then for $t\in C$ general we have $\dim(\overline{\phi_{|X_{t}}(X_{t})})\geq d_0$.

In particular, if there exists $t_0\in C$ such that $\phi_{|X_{t_0}}:X_{t_0}\dasharrow\mathbb{P}^n$ is generically finite, then for a general $t\in C$ the rational map $\phi_{|X_{t}}:X_{t}\dasharrow\mathbb{P}^n$ is generically finite as well.
\end{Proposition}
\begin{proof}
Let us consider the closure $Y = \overline{\phi(X)}\subseteq \mathbb{P}^n_C$ of the image of $X$ through $\phi$. By taking the restriction $\pi_{|Y}:Y\rightarrow C$ of the projection $\pi:\mathbb{P}^n_C\rightarrow C$ we see that $Y$ is a scheme over $C$.

Note that since $Y$ is an irreducible and reduced scheme over the curve $C$ we have that $Y$ is flat over $C$. In particular, the dimension of the fibers $\pi_{|Y}^{-1}(t)=Y_t$ is a constant $d = \dim(Y_t)$ for any $t\in C$.
  
For $t\in C$ general the fiber $\pi_{|Y}^{-1}(t)=Y_t$ contains $\phi_{|X_t}(X_t)$ as a dense subset. Therefore, we have $d = \dim(\phi_{|X_t}(X_t))\leq \dim(X_t)$ for $t\in C$ general.

Then, since $\phi_{|X_{t_0}}(X_{t_0})\subseteq Y_{t_0}$ we have $\dim(\overline{\phi_{|X_{t_0}}(X_{t_0})})\leq d = \dim(\overline{\phi_{|X_t}(X_t)})$ for $t\in C$ general.

Now, assume that $\dim(X_{t_0}) = \dim(\phi_{|X_{t_0}}(X_{t_0}))\leq d$. Therefore, we get 
$$\dim(X_{t_0})\leq d\leq \dim(X_t) = \dim(X_{t_0})$$ 
that yields $d = \dim(X_{t_0}) = \dim(X_t)$ for any $t\in C$. Hence, for a general $t\in C$ we have 
$$\dim(X_t) = \dim(\overline{\phi_{|X_t}(X_t)})$$
that is $\phi_{|X_{t}}:X_{t}\dasharrow \overline{\phi_{|X_t}(X_t)}\subseteq\mathbb{P}^n$ is generically finite.
\end{proof}

Now, let $C$ be a smooth and irreducible curve, $X\subset\mathbb{P}^N$ an irreducible and reduced projective variety, and $f:\Lambda\rightarrow C$ a family of $k$-dimensional linear subspaces of $\mathbb{P}^n$ parametrized by $C$. 

Let us consider the invertible sheaf $\mathcal{O}_{\mathbb{P}^n\times C}(1),$ and the sublinear system $|\mathcal{H}_{\Lambda}|\subseteq |\mathcal{O}_{\mathbb{P}^n\times C}(1)|$ given by the sections of $\mathcal{O}_{\mathbb{P}^n\times C}(1)$ vanishing on $\Lambda\subset\mathbb{P}^n\times C$. We denote by $\pi_{\Lambda|X\times C}$
the restriction of the rational map $\pi_{\Lambda}:\mathbb{P}^n\times C\dasharrow \mathbb{P}^{n-k-1}\times C$ of schemes over $C$ induced by $|\mathcal{H}_{\Lambda}|$.

Furthermore, for any $t\in C$ we denote by $\Lambda_t\cong\mathbb{P}^k$ the fiber $f^{-1}(t)$, and by $\pi_{\Lambda_t|X}$ the restriction to $X$ of the linear projection $\pi_{\Lambda_t}:\mathbb{P}^n\dasharrow\mathbb{P}^{n-k-1}$ with center $\Lambda_t$. 

\begin{Proposition}\label{p2}
Let $d_0 = \dim(\overline{\pi_{\Lambda_{t_0}|X}(X)})$ for $t_0\in C$. Then 
$$\dim(\overline{\pi_{\Lambda_{t}|X}(X)})\geq d_0$$
for $t\in C$ general.

Furthermore, if there exists $t_0\in C$ such that $\pi_{\Lambda_{t_0}|X}:X\dasharrow\mathbb{P}^{n-k-1}$ is generically finite then $\pi_{\Lambda_{t}|X}:X\dasharrow\mathbb{P}^{n-k-1}$ is generically finite for $t\in C$ general.
\end{Proposition}
\begin{proof}
The rational map $\pi_{\Lambda|X\times C}:X\times C\dasharrow \mathbb{P}^{n-k-1}\times C$ of schemes over $C$ is just the restriction of the relative linear projection $\pi_{\Lambda}:\mathbb{P}^n\times C\dasharrow \mathbb{P}^{n-k-1}\times C$ with center $\Lambda$. 

Therefore, the restriction of $\pi_{\Lambda|X\times C}$ to the fiber $X_t\cong X$ of $X\times C$ over $t\in C$ induces the linear projection from the linear subspace $\Lambda_t$, that is
$$\pi_{\Lambda|X_t} = \pi_{\Lambda_t|X}$$
for any $t\in C$.  Now, to conclude it is enough to apply Proposition \ref{p1} with $\phi = \pi_{\Lambda|X\times C}$.  
\end{proof}

Essentially, Propositions \ref{p1} and \ref{p2} say that the dimension of the general fiber of the special map is greater or equal than the dimension of the general fiber of the general map. Therefore, when the special map is generically finite the general one is generically finite as well. We would like to stress that in this case, under suitable assumptions, \cite[Lemma 5.4]{AGMO16} says that the degree of the map can only decrease under specialization.

\section{Non secant defectivity via osculating projections}\label{grassnodef}
In this section we use the techniques developed in Section \ref{degtanosc} to study the dimension of secant varieties of Grassmannians. Our first step consists in reinterpreting Proposition \ref{cc} in terms of osculating projections. In order to do this, we need to describe how many tangent spaces we can take in such a way that the flat limit of the span of them is contained in a higher order osculating space.

First, given an irreducible projective variety satisfying Assumption \ref{ass} and having $m$-osculating regularity, we introduce a function $h_m:\N_{\geq 0}\to \N_{\geq 0}$ counting how many tangent spaces we can degenerate to a higher order osculating space. 

\begin{Definition}\label{defhowmanytangent}
Given an integer $m\geq 2$ we define a function
\begin{align*}
h_m:\N_{\geq 0}\to \N_{\geq 0}
\end{align*}
as follows: $h_m(0)=0$. For any $k\geq 1$ write
$$k+1=2^{\lambda_1}+2^{\lambda_2}+\dots+2^{\lambda_l}+\varepsilon$$
where $\lambda_1>\lambda_2>\dots>\lambda_l \geq 1$, $\varepsilon\in \{0,1\}$, and define
$$h_m(k):=m^{\lambda_1-1}+m^{\lambda_2-1}+\dots+m^{\lambda_l-1}.$$
In particular $h_m(2k)=h_m(2k-1)$ and $h_2(k)=\left\lfloor \dfrac{k+1}{2}\right\rfloor$.
\end{Definition}

\begin{Example}
For instance 
$$h_m(1)=h_m(2)=1,h_m(3)=m,h_m(5)=m+1,h_m(7)=m^2,h(9)=m^2+1$$
and since $23=16+4+2+1=2^4+2^2+2^1+1$ we have $h_m(22)=m^3+m+1$.
\end{Example}

\begin{Theorem}\label{lemmadefectsviaosculating}
Let $X\subset \P^N$ be an irreducible projective variety satisfying Assumption \ref{ass} and
having $m-$osculating regularity, $p_1,\dots,p_l\in X$ general points, $k_1,\dots,k_l\geq 1$ integers, and set 
$$h:= \displaystyle\sum_{j=1}^{l}h_m(k_j).$$
If $\Pi_{T^{k_1,\dots,k_l}_{p_1,\dots,p_l}}$ is generically finite then $X$ is not $(h+1)$-defective.
\end{Theorem}
\begin{proof}
Let us consider a general tangential projection $\Pi_{T}$ where 
$$T = \left\langle T^1_{p_1^1},\dots, T^1_{p_1^{h_m(k_1)}},\dots, T^1_{p_l^1},\dots,T^1_{p_l^{h_m(k_l)}}\right\rangle$$
and $p_1^1 = p_1,\dots, p_l^1 = p_l$. Our argument consists in specializing the projection $\Pi_{T}$ several times in order to reach a generically finite projection. For seek of notational simplicity along the proof we will assume $l = 1$. For the general case it is enough to apply the same argument $l$ times.

Let us begin with the case $k_1+1 = 2^{\lambda}$. Then $h_{m}(k_1) = m^{\lambda-1}$. Since $X$ has $m$-osculating regularity we can degenerate $\Pi_{T}$, in a family parametrized by a smooth curve, to a projection $\Pi_{U_1}$ whose center $U_1$ is contained in 
$$V_1 = \left\langle T^{3}_{p_1^1}, T^3_{p_1^{m+1}},\dots, T^3_{p_1^{m^{\lambda-1}-m+1}}\right\rangle.$$ 
Again, since $X$ has $m$-osculating regularity we may specialize, in a family parametrized by a smooth curve, the projection $\Pi_{V_1}$ to a projection $\Pi_{U_2}$ whose center $U_2$ is contained in
$$V_2 = \left\langle T^{7}_{p_1^1}, T^7_{p_1^{m^2+1}},\dots, T^7_{p_1^{m^{\lambda-1}-m^2+1}}\right\rangle.$$
Proceeding recursively in this way in last step we get a projection $\Pi_{U_{\lambda-1}}$ whose center $U_{\lambda -1}$ is contained in 
$$V_{\lambda-1} = T^{2^{\lambda}-1}_{p_1^1}.$$
When $k_1+1 = 2^{\lambda}$ our hypothesis means that $\Pi_{T^{k_1}_{p_1^1}}$ is generically finite. Therefore, $\Pi_{U_{\lambda-1}}$ is generically finite, and applying Proposition \ref{p2} recursively to the specializations in between $\Pi_{T}$ and $\Pi_{U_{\lambda-1}}$ we conclude that $\Pi_{T}$ is generically finite as well.

Now, more generally, let us assume that 
$$k_1+1 = 2^{\lambda_1}+\dots + 2^{\lambda_s}+\varepsilon$$
with $\varepsilon\in\{0,1\}$, and $\lambda_1 > \lambda_2 > \dots > \lambda_s\geq 1$. Then
$$h_m(k_1) = m^{\lambda_1-1}+\dots + m^{\lambda_s-1}.$$
By applying $s$ times the argument for $k_1+1 = 2^{\lambda}$ in the first part of the proof we may specialize $\Pi_{T}$ to a projection $\Pi_{U}$ whose center $U$ is contained in 
$$V = \left\langle T^{2^{\lambda_1}-1}_{p_1^1}, T^{2^{\lambda_2}-1}_{p_1^{m^{\lambda_1-1}+1}},\dots, T^{2^{\lambda_s}-1}_{p_1^{m^{\lambda_1-1}+\dots+m^{\lambda_{s-1}-1}+1}}\right\rangle.$$ 
Finally, we use Assumption \ref{ass} $s-1$ times to specialize $\Pi_V$ to a projection $\Pi_{U^{'}}$ whose center $U^{'}$ is contained in 
$$V^{'} = T_{p_1^1}^{2^{\lambda_1}+\dots +2^{\lambda_s}-1}.$$
Note that $T_{p_1^1}^{2^{\lambda_1}+\dots +2^{\lambda_s}-1} = T^{k_1}_{p_1^1}$ if $\varepsilon = 0$, and $T_{p_1^1}^{2^{\lambda_1}+\dots +2^{\lambda_s}-1} = T^{k_1-1}_{p_1^1}\subset T^{k_1}_{p_1^1}$ if $\varepsilon = 1$. In any case, since by hypothesis $\Pi_{T^{k_1}_{p_1^1}}$ is generically finite, again by applying Proposition \ref{p2} recursively to the specializations in between $\Pi_{T}$ and $\Pi_{U^{'}}$ we conclude that $\Pi_{T}$ is generically finite. Therefore, by Proposition \ref{cc} we get that $X$ is not $(\sum_{j=1}^{l}h_m(k_j)+1)$-defective.
\end{proof}

Now, we are ready to prove our main result on non-defectivity of Grassmannians.

\begin{Theorem}\label{maingrass}
Assume that $r\geq 2$, set 
$$\alpha:=\left\lfloor \dfrac{n+1}{r+1} \right\rfloor$$
and let $h_\alpha$ be as in Definition \ref{defhowmanytangent}. If either
\begin{itemize}
	\item[-] $n\geq r^2+3r+1$ and $h\leq\alpha h_{\alpha}(r-1)$ or
	\item[-] $n< r^2+3r+1$, $r$ is even, and 
	$h\leq (\alpha-1) h_{\alpha}(r-1)+
	h_\alpha(n-2-\alpha r)$ 	or
	\item[-] $n< r^2+3r+1$, $r$ is odd, and 
	$h\leq (\alpha-1) h_{\alpha}(r-2)+h_\alpha(\min\{n-3-\alpha(r-1),r-2\})$
\end{itemize}
then $\G(r,n)$ is not $(h+1)$-defective.
\end{Theorem}
\begin{proof}
Since by Propositions \ref{limitosculatingspacesgrass} and \ref{limitosculatingspacesgrassII}
the Grassmannian $\G(r,n)$ satisfies Assumption \ref{ass} and has $\alpha-$osculating regularity,
it is enough to apply Corollary \ref{oscprojbirationalII} together with Theorem \ref{lemmadefectsviaosculating}.
\end{proof}
Note that if we write 
\begin{equation}\label{222}
r = 2^{\lambda_1}+2^{\lambda_2}+\dots + 2^{\lambda_s}+\varepsilon
\end{equation}
with $\lambda_1 >\lambda_2>\dots >\lambda_s\geq 1$, $\varepsilon\in\{0,1\}$, then 
$$h_{\alpha}(r-1) = \alpha^{\lambda_1-1}+\dots + \alpha^{\lambda_s-1}.$$ 
Therefore, the first bound in Theorem \ref{maingrass} gives 
$$h\leq \alpha^{\lambda_1}+\dots + \alpha^{\lambda_s}.$$
Furthermore, just considering the first summand in the second and third bound in Theorem \ref{maingrass} we get that $\G(r,n)$ is not $(h+1)$-defective for
$$h\leq (\alpha-1)(\alpha^{\lambda_1-1}+\dots + \alpha^{\lambda_s-1}).$$ 
Finally, note that (\ref{222}) yields $\lambda_1 = \lfloor\log_2(r)\rfloor$. Hence, asymptotically we have $h_{\alpha}(r-1)\sim \alpha^{\lfloor\log_2(r)\rfloor-1}$, and by Theorem \ref{maingrass} $\G(r,n)$ is not $(h+1)$-defective for 
$$h\leq \alpha^{\lfloor\log_2(r)\rfloor} = \left(\frac{n+1}{r+1}\right)^{\lfloor\log_2(r)\rfloor}.$$

\begin{Example}
%If $r=3$, $n=11$ then $\alpha=3$, $r'=0$, $r''=2$, and by Corollary \ref{oscprojbirationalII} we have that $\Pi_{T_{e_{I_1},e_{I_2}}^{2,2}}$, $\Pi_{T_{e_{I_1},e_{I_2},e_{I_3}}^{2,2,0}}$ and $\Pi_{T_{e_{I_1},e_{I_2},e_{I_3}}^{1,1,2}}$ are birational. Moreover, by Theorem \ref{maingrass}, $\G(3,11)$ is not $h-$defective for any $h\leq 4$.

%Similarly, if $r=5$, $n=24$ then $\alpha=4$, $r'=2$, $r''=4$, and by Corollary \ref{oscprojbirationalII} 
%we have that $\Pi_{T_{e_{I_1},e_{I_2},e_{I_3}}^{4,4,4}}$, $\Pi_{T_{e_{I_1},e_{I_2},e_{I_3},e_{I_4}}^{4,4,4,2}}$ and $\Pi_{T_{e_{I_1},e_{I_2},e_{I_3},e_{I_4}}^{3,3,3,4}}$ are birational. Furthermore, by Theorem \ref{maingrass} $\G(5,23)$ is not $h-$defective for any $h\leq 17$.

In order to help the reader in getting a concrete idea of the order of growth of the bound in Theorem \ref{maingrass} for $n\geq r^2+3r+1$ we work out some cases in the following table:
\begin{center}
\begin{tabular}{|c|c|l|}
\hline 
$r$ & $r^2+3r+1$ & $h$\\ 
\hline 
$4$ & $29$ & $\left(\frac{n+1}{5}\right)^2+1$\\ 
\hline 
$6$ & $55$ & $\left(\frac{n+1}{7}\right)^2+\left(\frac{n+1}{7}\right)+1$\\ 
\hline 
$8$ & $89$ & $\left(\frac{n+1}{9}\right)^3+1$\\ 
\hline 
$10$ & $131$ & $\left(\frac{n+1}{11}\right)^3+\left(\frac{n+1}{11}\right)+1$\\ 
\hline 
$12$ & $181$ & $\left(\frac{n+1}{13}\right)^3+\left(\frac{n+1}{13}\right)^2+1$\\ 
\hline 
$14$ & $239$ & $\left(\frac{n+1}{15}\right)^3+\left(\frac{n+1}{15}\right)^2+\left(\frac{n+1}{15}\right)+1$\\ 
\hline
$16$ & $305$ & $\left(\frac{n+1}{17}\right)^4+1$\\
\hline 
\end{tabular} 
\end{center}
\end{Example}

Thanks to Theorem \ref{maingrass} it is straightforward to get a linear bound going with $\frac{n}{2}$.

\begin{Corollary}\label{maincor}
Assume that $r\geq 2$, and set 
$$\alpha:=\left\lfloor \dfrac{n+1}{r+1} \right\rfloor$$
If either
\begin{itemize}
	\item[-] $n\geq r^2+3r+1$ and $h\leq\left\lfloor \dfrac{r}{2} \right\rfloor\alpha+1$ or
	\item[-] $n< r^2+3r+1$, $r$ is even, and 
	$h\leq	\left\lfloor \dfrac{n+1}{2} \right\rfloor-\dfrac{r}{2}$ or
	\item[-] $n< r^2+3r+1$, $r$ is odd, and $h\leq \min
	\left\{ \dfrac{r-1}{2} \alpha+1,\left\lfloor \dfrac{n}{2} \right\rfloor-\dfrac{r-1}{2}\right\}$
\end{itemize}
then $\G(r,n)$ is not $h$-defective.
\end{Corollary}
\begin{proof}
Since $\alpha \geq 2$ we have $h_\alpha(k)\geq h_2(k)=\left\lfloor \dfrac{k+1}{2} \right\rfloor$. In particular, $h_\alpha(r-1)\geq \left\lfloor \dfrac{r}{2} \right\rfloor$ and $h_\alpha(r-2)\geq \left\lfloor \dfrac{r-1}{2} \right\rfloor$.

Now, it is enough to observe that
\begin{align*}
\dfrac{r}{2} (\alpha-1)+\left\lfloor \dfrac{n-2-\alpha r+1}{2} \right\rfloor+1=
\left\lfloor \dfrac{n+1}{2} \right\rfloor-\dfrac{r}{2}
\end{align*}
for $r$ even, and 
\begin{align*}
\dfrac{r-1}{2} (\alpha-1)+\left\lfloor \dfrac{n-3-\alpha (r-1)+1}{2} \right\rfloor+1=
\left\lfloor \dfrac{n}{2} \right\rfloor-\dfrac{r-1}{2}
\end{align*}
for $r$ odd, and to apply Theorem \ref{maingrass}.
\end{proof}

\subsection{Comparison with Abo-Ottaviani-Peterson bound}\label{uglymath}
Finally, we show that Corollary \ref{maincor} strictly improves \cite[Theorem 3.3]{AOP09b} for $r\geq 4$, whenever $(r,n)\notin\left\{(4,10),(5,11)\right\}$.

For $r\geq 4,n\geq 2r+1$ we define the following functions of $r$ and $n$:
\begin{align*}
a:=\left\lfloor \dfrac{r}{2} \right\rfloor\left\lfloor \dfrac{n+1}{r+1} \right\rfloor,\ 
a':=\left\lfloor \dfrac{n-1}{2} \right\rfloor-\dfrac{r}{2},\ 
a'':=\left\lfloor \dfrac{n}{2} \right\rfloor-\dfrac{r+1}{2},\ 
b:=\left\lfloor \dfrac{n-r}{3} \right\rfloor
\end{align*}
First we show that $a>b$. Indeed, if $r>2$ is even then
$$a=\dfrac{r}{2}\left\lfloor \dfrac{n+1}{r+1} \right\rfloor>\dfrac{r}{2}\cdot\dfrac{n-r}{r+1}>
\dfrac{n-r}{3}\geq\left\lfloor \dfrac{n-r}{3} \right\rfloor=b$$
and if $r>5$ is odd then
$$a=\dfrac{r-1}{2}\left\lfloor \dfrac{n+1}{r+1} \right\rfloor>\dfrac{r-1}{2}\cdot\dfrac{n-r}{r+1}>
\dfrac{n-r}{3}\geq\left\lfloor \dfrac{n-r}{3} \right\rfloor=b.$$
Furthermore, if $r=5$ we write $n=6\lambda+\varepsilon$ with $\varepsilon\in\{-1,0,1,2,3,4\}$. Then we have
$$a=2\left\lfloor \dfrac{6\lambda+\varepsilon+1}{6} \right\rfloor=2\lambda>
2\lambda+\left\lfloor \dfrac{\varepsilon-5}{3} \right\rfloor=\left\lfloor \dfrac{6\lambda+\varepsilon-5}{3} \right\rfloor=b.$$

Now, we assume that $n<r^2+3r+1$ and we show that $a'>b$ if $r$ is even and $(r,n)\neq (4,10)$, and that $a''>b$ if $r$ is odd and $(n,r)\neq (5,11)$. Note that $a'(4,10)=a''(5,11)=b(4,10)=b(5,11)=2$. If $r$ is even
$$a'=\left\lfloor \dfrac{n-1}{2} \right\rfloor-\dfrac{r}{2}>
\dfrac{n-1}{2}-1-\dfrac{r}{2}=\dfrac{n-r-3}{2}>\dfrac{n-r}{3}=b$$
whenever $n>r+9.$ Similarly, if $r$ is odd and $n>r+9$ we have 
$$a''=\left\lfloor \dfrac{n}{2} \right\rfloor-\dfrac{r+1}{2}>
\dfrac{n}{2}-1-\dfrac{r+1}{2}=\dfrac{n-r-3}{2}>\dfrac{n-r}{3}=b.$$
Now, if $r>8$ then $n\geq 2r+1\Rightarrow n>r+9$. A finite number of cases are left, namely
$$(r,n)\in\left\{(r,n);\: r=4,5,6,7,8 \mbox{ and } r^2+3r+1>r+9\geq n \geq 2r+1 \right\}$$
These cases can be easily checked one by one.

%\begin{Remark}
%Finally, we would like to stress that the bound in Theorem \ref{main} may be improved in two ways. First, one could strengthen Corollary \ref{oscprojbirationalII} considering osculating projection from more than $\alpha = \lfloor\frac{n+1}{r+1}\rfloor$ osculating spaces. Second, one could improve Proposition \ref{limitosculatingspacesgrassII} on the $m$-osculating regularity of $\G(r,n)$. This last one would give an improvement on the order of growth of the bound. 
%\end{Remark}

\end{document}